\documentclass[reqno]{amsart}

\usepackage{amsmath}
\usepackage{amsfonts}
\usepackage{amsthm}
\usepackage{amssymb}
\usepackage{graphicx}
\usepackage{graphics}
\usepackage{bm}
\usepackage{dsfont}
\usepackage{color}
\usepackage{enumerate}
\usepackage[font=footnotesize]{caption}
\usepackage[all]{xy}
\usepackage{xypic}
\usepackage{tikz}
\usepackage{paralist}
\allowdisplaybreaks
\usepackage{pgfplots}
\pgfplotsset{%
   every tick label/.append style = {font=\tiny},
   every axis label/.append style = {font=\scriptsize}
}

\numberwithin{equation}{section}
\newtheorem{thm}{Theorem}[section]
\newtheorem*{thm*}{Theorem}
\newtheorem{cor}[thm]{Corollary}
\newtheorem{lem}[thm]{Lemma}
\newtheorem{prop}[thm]{Proposition}
\theoremstyle{definition}

\newtheorem{qst}[thm]{Question}

\newtheorem{dfn}[thm]{Definition}
\newtheorem{rmk}[thm]{Remark}
\newcommand{\N}{\mathds{N}}

\newcommand{\Z}{\mathds{Z}}

\newcommand{\R}{\mathds{R}}

\newcommand{\T}{\mathds{T}}

\newcommand{\diff}{\mathrm{d}}

\newcommand{\U}{\mathcal{U}}
\newcommand{\W}{\mathcal{W}}

\newcommand{\V}{\mathcal{V}}

\newcommand{\sss}{\mathrm s}
\newcommand{\llll}{\mathrm l}
\newcommand{\iii}{\mathrm i}

\newcommand{\A}{\mathbb{A}}

\newcommand{\dist}{\mathrm{dist}}

\newcommand{\fl}{\mathrm{flat}}

\newcommand{\con}{\mathrm{con}}

\newcommand{\kin}{{\mathrm{kin}}}

\begin{document}

\title[On the rigidity of Zoll magnetic systems on surfaces]{On the rigidity of Zoll magnetic systems on surfaces}

\author[L. Asselle]{Luca Asselle}
\address{Justus Liebig Universit\"at Giessen, Mathematisches Institut, Arndtstrasse 2, \newline\indent Raum 102, D-35392 Giessen, Germany}
\email{luca.asselle@ruhr-uni-bochum.de}

\author[C. Lange]{Christian Lange}
\address{Mathematisches Institut der Universit\"at K\"oln, Weyertal 86-90, Raum -103, 50931, K\"oln, Germany}
\email{clange@math.uni-koeln.de}

\date{July, 2019}
\subjclass[2000]{37J99, 58E10}
\keywords{Magnetic flows, Zoll systems, waists}

\begin{abstract}
In this paper we study rigidity aspects of Zoll magnetic systems on closed surfaces. We 
characterize magnetic systems on surfaces of positive genus given by constant curvature metrics and 
constant magnetic functions as the only magnetic systems such that the associated Hamiltonian flow is Zoll, i.e. every orbit is closed, on every energy level. We also prove the persistence of possibly degenerate closed geodesics under magnetic perturbations in different instances.
\end{abstract}

\maketitle

%%%%%%%%%%%
\section{Introduction}

Let $\Sigma$ be a closed oriented surface. A \textit{magnetic system} on $\Sigma$ is a pair $(g,f)$, where $g$ is a Riemannian metric on $\Sigma$ and $f:\Sigma \to \R$ is a smooth function (the \textit{magnetic function}). Every magnetic system defines a flow on $S\Sigma$, the unit tangent bundle 
of $\Sigma$, as we now briefly recall: a smooth arc-length parametrized curve $\gamma:I\rightarrow \Sigma$ is called a $(g,f)$-\textit{geodesic}, 
if it has geodesic curvature equal to $f$, that is, if it satisfies 
\begin{equation}
\nabla_{\dot \gamma} \dot \gamma = (f\circ \gamma) \cdot \dot \gamma^\perp,
\label{prescribedcurvature}
\end{equation}
where $\nabla$ is the Levi-Civita connection, and $\dot \gamma^\perp$ is the unit tangent vector such that the angle between $\dot \gamma$ and 
$\dot \gamma^\perp$ is $\frac \pi 2$ (recall that $g$ and the fixed orientation yield a well-defined way of measuring angles in each tangent plane, as well as an area form $\mu_g$ on $\Sigma$). The flow on $S \Sigma$ is given by 
$$\phi_{g,f}^t (q,v) = (\gamma(t),\dot \gamma(t)), \quad \forall t\in \R,$$
where $\gamma$ is the unique solution to \eqref{prescribedcurvature} with $\gamma(0)=q$, $\dot \gamma(0)=v$. Such a flow is %generated 
%by the vector field
%$$X_{g,f} := X_g + \frac{1}{2\pi} (f\circ \pi) V,$$
%where $X_g$ is the geodesic vector field of $g$, and $V$ is the vector field whose flow rotates the tangent fibres in the 
%$\mathfrak o$-negative direction with constant angular speed $\frac{1}{2\pi}$, and 
of physical interest since it models the motion of a charged particle in $\Sigma$ under the effect of the 
magnetic field $f\mu_g$. The Legendre transform provides a conjugacy between $\Phi^t_{g,f}$ and 
the (restriction to the energy level $\frac 12$ of the) magnetic flow, that is the Hamiltonian flow on $T^*\Sigma$ defined by
$H_\kin(q,p)=\frac 12 |p|^2$ and the twisted symplectic form 
$$\omega_{g,f} := \diff p \wedge \diff q - \pi^* (f \mu_g),$$ 
where $\pi:T^*\Sigma\to \Sigma$ 
is the bundle projection. Moreover, for every $\lambda >0$ the reparametrization $\tilde \gamma(t):=\gamma(t/\lambda)$ 
yields a correspondence between $(g,\lambda f)$-geodesics and orbits of the magnetic flow
contained in $\{H_\kin=\frac{1}{2\lambda^2}\}$. Hence, the magnetic flow can be seen as the collection $\{\Phi^t_{g,\lambda f}\}$ of flows  associated with the 
family of magnetic systems $\{(g,\lambda f)\, |\, \lambda >0\}$; also we have a correspondence between high (low) energies and small (large) values of $\lambda$.
After the pioneering work of Arnol'd \cite{Arnold:1961} in 1960s, magnetic systems have received the attention of many outstanding mathematicians, such as Novikov, Ginzburg, 
and Contreras, among others. In particular, the problem of finding periodic solutions to \eqref{prescribedcurvature}, which 
we will refer to as \textit{closed} $(g,f)$-\textit{geodesics}, turned out to be extremely difficult, and many questions in the topic still remain open or 
only partially answered. We refer the reader to \cite{Asselle:2014hc,Contreras:2004lv,Ginzburg:1994,Taimanov:1992sm} and references therein 
for an account of the main contributions to the closed $(g,f)$-geodesics problem, particularly for the case of surfaces. 
We shall recall that, in contrast with geodesic flows, magnetic systems 
present very different behaviors for different values of $\lambda$; see e.g. \cite{Abbondandolo:2013is,Contreras:2006yo}. 

In this paper we will focus on the complementary problem, that is, in the study of systems $(g,f)$ whose flow \eqref{prescribedcurvature} is orbit-equivalent to a free $S^1$-action on $S\Sigma$, hence in particular for which all orbits are closed.

\begin{dfn}
A magnetic system $(g,f)$ is called \textit{Zoll} if  $\Phi^t_{g,f}$ has the same orbits as a free $S^1$-action on $S\Sigma$.
\end{dfn}

\begin{rmk}
\label{rmk:allclosed}
For $f\equiv 0$ we recover the notion of a Zoll metric, and in this case $\Sigma$ must be the two-sphere (a thorough discussion of such metrics can be 
found e.g. in \cite{Besse}). For $\Sigma\neq S^2$, Zoll magnetic systems can be equivalently defined as those pairs $(g,f)$ for which all
$(g,f)$-geodesics are closed (and contractible). Indeed, if all $(g,f)$-geodesics are closed, then $\Phi^t_{g,f}$ has by a theorem of Epstein \cite{Epstein:1972}
the same orbits as a fixed-point free $S^1$-action on $S\Sigma$, and hence defines a Seifert fibration of $S\Sigma$. 
It follows from \cite[Theorem 5.1]{Jankins} that this must be the standard, regular $S^1$-fibration (given by $S\Sigma \rightarrow\Sigma$) . 
For $\Sigma=S^2$, the celebrated Katok's example \cite{Benedetti:2016,Katok:1973mw} yields magnetic systems 
$(g,f)$ for which all $(g,f)$-geodesics are closed but whose flow on $S\Sigma$ is only orbit-equivalent to a semi-free $S^1$-action.\qed
\end{rmk}

Zoll magnetic systems exist for every closed oriented surface and, as established in \cite{Benedetti:2018c}, play a crucial role in local magnetic systolic inequalities: 
the minimal magnetic length of closed magnetic geodesics of magnetic systems close to a Zoll one is bounded from above in terms 
of a quantity depending only on the $g$-volume, the genus of $\Sigma$, and the integral of $f$ over $\Sigma$, and the upper bound is attained precisely when 
the magnetic system is Zoll. Also, recently the first author and Benedetti \cite{Asselle:2019a} showed that integrable magnetic systems on the 
two-torus admitting a global surface of section satisfy a sharp systolic inequality (see \cite{Abbondandolo:2018} for a similar result for Riemannian spheres of revolution). 
For the applications of such systolic inequalities it is therefore crucial to gain a better understanding of the space of Zoll magnetic systems.

Until last year, the only known examples of Zoll magnetic systems were pairs $(g_\con,f_\con)$,
with $g_\con$ a metric of constant curvature $K_\con$ and $f_\con>0$ any constant function such that\footnote{Taking $f_\con=K_\con=0$ for the two-torus yields the geodesic flow of the flat metric, and in particular all closed $(g_\con,f_\con)$-geodesics
are not contractible. A similar situation occurs for $f_\con^2+K_\con<0$ on higher genus surfaces, whereas for $f_\con^2+K_\con=0$ 
we retrieve the celebrated \textit{horocycle flow} of Hedlund (cf. \cite{Hedlund:1932am}), and in particular no closed $(g_\con,f_\con)$-geodesics at all.} 
$$f_\con^2 + K_\con >0.$$
%Indeed, the lift $\tilde \gamma$ to the universal cover of a prime closed $(g_\con,f_\con)$-geodesic $\gamma$ parametrizes the boundary of a geodesic ball of radius 
%depending only on $K_\con$ and $f_\con$,
%$$R = \left \{ \begin{array}{r} \frac{1}{\sqrt{K_*}} \arctan \left (\frac{\sqrt{K_*}}{f_*}\right ) \quad \quad \ \ K_*>0,\\
%						\frac{1}{f_*} \quad \quad \quad \quad \quad \quad \quad \quad K_*=0,\\
%						\frac{1}{\sqrt{-K_*}} \text{arctanh} \left (\frac{\sqrt{-K_*}}{f_*}\right ) \quad K_*<0,
%						\end{array}\right .$$
%and the map $\mathfrak p:S\Sigma\to \Sigma$ associates to a tangent vector the projection on $\Sigma$ of the center of the corresponding ball in $\widetilde \Sigma$.
A breakthrough came only very recently with \cite{Asselle:2019a}, in which explicit
non-trivial 1-parameter families of rotationally symmetric Zoll magnetic systems on certain flat tori are constructed. 
The result in \cite{Asselle:2019a} can be thought of as the first evidence of the flexibility of magnetic flows which are Zoll at a given energy. 
However, the trivial examples are Zoll at \textit{every} energy, that is, $(g_\con,\lambda f_\con)$ is Zoll for every $\lambda >0$ such that 
$\lambda^2f_\con^2 + K_\con>0$. Therefore, since magnetic flows strongly depend on the energy, it is natural to ask the following 

\begin{qst} \label{qst:flex_rigid}
Does the flexibility in \cite{Asselle:2019a} turn into rigidity, if one requires the magnetic flow to be Zoll at multiple energies?
\end{qst}

In order to make the question more precise, we recall that 
if $\Sigma$ has genus greater than or equal to $2$, then the two-form $f\mu_g$ is weakly-exact, 
that is, its lift to the universal cover is exact. 
We fix a primitive $\theta$ of the lift of $f\mu_g$ and define 
the \textit{Ma\~n\'e critical value of the universal cover} as 
\begin{equation}
c(g,f) := \inf_{u\in C^\infty(\tilde \Sigma)} \sup_{q\in \tilde \Sigma} \frac 12 |\theta_q - \diff_q u|^2,
\label{mane}
\end{equation}
where $|\cdot|$ denotes the dual norm on $T^*\tilde \Sigma$ induced by the lift of the metric $g$. It is well-known that $c(g,f)$ is always finite (see e.g. \cite{Asselle:2015ij}),
and vanishes if and only if $f\equiv 0$. Now we set
\begin{equation}
h(g,f) := \frac{1}{\sqrt{2c(g,f)}} \in (0,+\infty).
\label{oneovermane}
\end{equation}
We shall notice that  $c(g,f)$ is well defined also for the two-torus, and for any surface if $f\mu_g$ is exact. 
However, as we will see, for our purposes we can always assume that $f\mu_g$ is not exact. As it turns out, in this case
$c(g,f)$ is always infinite if $\Sigma$ is a two-torus, 
and if $\Sigma=S^2$ we have that any primitive of $f\mu_g|_{S^2\setminus \{p\}}$ is unbounded. 
Therefore, if $\Sigma$ is a two-sphere or a two-torus we set 
$$h(g,f):=0.$$ Also, if $\Sigma$ is a surface with genus at least two, then for every $\lambda < h(g,f)$ there exists 
a closed $(g,\lambda f)$-geodesic in every non-trivial free homotopy class; see \cite{Merry:2010}. In particular, $(g,\lambda f)$ cannot be Zoll
for $\lambda < h(g,f)$.

The main goal of this paper is to provide a positive answer to Question \ref{qst:flex_rigid} in the following form.

\begin{thm}
Let $(g,f)$ be a magnetic system on a surface with genus greater than or equal to one such that $(g,\lambda_n f)$ is Zoll for 
some bi-infinite sequence $\{\lambda_n\}_{n\in \Z}$ with $\lambda_n \downarrow h(g,f)$ for $n\to -\infty$ and $\lambda_n\uparrow +\infty$ for $n\to +\infty$. 
Then $g$ has constant curvature and $f$ is constant.
\label{thm:main}
\end{thm}

We now give an account on the strategy of the proof of Theorem \ref{thm:main} for the two-torus, which throughout the paper will be identified with
the quotient $\T^2$ of $\R^2$ by the lattice $\Z\times \Z$:

\begin{itemize}
\item The decreasing sequence $\lambda_n \downarrow 0$ yields that $g$ is flat. This follows from the fact that closed non-contractible geodesics 
which globally minimize the length in their homotopy class cannot all disappear under magnetic perturbations, unless the metric is flat.
This fact will be proved in Section \ref{section:2}.
\item Once we know that the metric is flat, we look at the magnetic systems $(g_\fl,\lambda_n f)$ for $n\to +\infty$. In Section \ref{section:3}
we show that, if $f$ is not constant, then for $n$ sufficiently large we find both short and long closed $(g_\fl,\lambda_n f)$-geodesics. 
The fact that $(g_\fl,\lambda_n f)$ is Zoll then contradicts the dichotomy between short and long periodic orbits established in \cite[Theorem 7.13]{Benedetti:diss}.
\end{itemize}

The proof for higher genus surfaces is different and hinges on the relation between the helicity and the Ma\~n\'e critical value $c(g,f)$
proved in \cite{Paternain:2009}; see the end of Section \ref{section:3}. Actually, in this case  
we only need the sequence $\lambda_n \downarrow h(g,f)$ to conclude rigidity; see Theroem \ref{thm:highergenus}. Therefore, we are prompted to 
ask the following

\begin{qst}
For $\Sigma$ a two-torus, does Theorem \ref{thm:main} continue to hold if one only requires that $(g,\lambda_n f)$ be Zoll for some decreasing sequence
$\lambda_n\downarrow 0$? More generally, can one detect the precise threshold between rigidity and flexibility? Does this threshold depend 
on the topology of the surface?
\end{qst}

In case of the two-sphere, a statement in the spirit of Theorem \ref{thm:main} turns out to be more difficult to prove, the main reason being that the space of Zoll metrics on the two-sphere is infinite dimensional \cite{Guillemin:1976}. Nevertheless, in Section \ref{section:4} we prove some partial results in this direction: we show that the metric $g$ must ``generically'' be Zoll and that genericity can be dropped if one considers only rotationally invariant magnetic systems.  Further computations with rotationally invariant magnetic systems also support rigidity on $S^2$ which is why we make the following 
\newline
\newline
\textbf{Conjecture (Z).} Let $(g,f)$ be a magnetic system on $S^2$ such that $(g,\lambda f)$ is Zoll for all
$\lambda >0$. Then $g$ has constant curvature and $f$ is constant. 
\vspace{2mm}

In fact, after finishing this paper, we were able to confirm Conjecture (Z) for rotationally invariant magnetic systems of the form $(g,f_\con)$. The proof will appear in a forthcoming paper.

Let us also mention that Theorem \ref{thm:main} and Conjecture (Z) are related to two open problems about integrable dynamical systems on two-dimensional configuration spaces. 
The first one traces back to Birkhoff \cite{Birkhoff:1966xb} and aims at determining all metrics on the two-torus with an integrable geodesic flow. 
Despite several partial results (see \cite{Bialy:2011,Bialy:2015,Kozlov:1989} and references therein), it is as of now not known whether there 
are metrics other than Liouville metrics which gives rise to integrable geodesic flows. 
The second problem concerns exact magnetic flows on the two-torus that admit a first integral on all energy levels: In \cite{Agapov:2017} it is conjectured that such 
flows must be of a very particular type (cf. Example 1 in \cite{Agapov:2017}), and the conjecture is confirmed in the case of quadratic in momenta integrals.

 %, and use this to confirm the validity of Conjecture (Z) for rotationally invariant magnetic systems of the 
%form $(g, f_\con)$ in the analytic setting. 
%
%\begin{thm}
%Let $g$ be a rotationally invariant analytic metric on $S^2$ such that the magnetic system $(g,\lambda f_\con)$ is Zoll for all $\lambda >0$. Then $g$ is the round metric. 
%\label{thm:mainsphere}
%\end{thm}

%Here analyticity is rather a convenient assumption which simplfies the proof, see Remark
%
%Our proof of this result illustrates how non-constant curvature affects the behaviour of a charged particle in the presence of strong magnetic fields. 

Finally, in Section \ref{section:5} we prove a result of independent interest on the persistence of possibly degenerate closed geodesics under magnetic perturbations which we 
can formulate roughly speaking as follows: Let $g$ be a metric admitting a closed contractible geodesic which is a local minimizer of the length functional and is stable (namely, does not disappear after an arbitrarily small perturbation of the metric). Then such a closed geodesic will persist also under magnetic perturbations.
One major issue we have to overcome in the proof of such a statement is that the dynamics of a magnetic systems arising as perturbation of 
a geodesic systems is in general drastically different from the geodesic dynamics, even if the perturbation is arbitrarily small.  Also, 
we have to deal with critical sets which may have complicated topology (such as e.g. a Cantor set).

\begin{thm}
Let $\Sigma$ be a closed orientable surface, and let $(g,f)$ be a magnetic system on $\Sigma$. Suppose that $g$ possesses a contractible stable waist, that is, a 
closed geodesic
that locally minimizes the free-period action functional $\A$ in \eqref{freeperiod1}. Then there exists $\Lambda(g,f)>0$ such that 
for all $0<\lambda<\Lambda (g,f)$ there exists a closed contractible $(g,\lambda f)$-geodesic which locally minimizes the free-period Lagrangian action functional
$\A^\lambda$ in \eqref{freeperiod2}. Moreover, such closed $(g,\lambda f)$-geodesics can be chosen to lie in a small neighborhood of a waist for $g$. 
\label{thm:main2}
\end{thm}

In fact, in our proof of Theorem \ref{thm:main2} the contractibilty assumption is only used in the case of the two-torus.
\newline
\newline
\textbf{Acknowledgments.} We warmly thank Alberto Abbondandolo and Stefan Suhr for many fruitful discussions. We are 
indebted to Gabriele Benedetti for suggesting us the reference \cite{Paternain:2009}. 
L.A. is partially supported by the DFG-grant AS 546/1-1 ``Morse theoretical
methods in Hamiltonian dynamics''. C.L. is partially supported by the DFG-grant SFB/TRR 191 ``Symplectic structures in Geometry, Algebra and Dynamics''.

%%%%%%%%%%%

\section{Zoll magnetic systems on $\T^2$ for small values of the parameter $\lambda$}
\label{section:2}

In this section we want to derive conditions on the metric $g$ for magnetic systems $(g,f)$ on $\T^2$ such that $(g,\lambda f)$ is 
Zoll for $\lambda>0$ sufficiently small, or, equivalently, such that the corresponding magnetic flow is Zoll for sufficiently large energies.
More precisely, we want to show that being Zoll for small values of $\lambda$ implies that the metric is flat. As it turns out, we 
don't need to require that $(g,\lambda f)$ is Zoll for \textit{all} $\lambda>0$ sufficiently small. Indeed, it is enough that $(g,\lambda_n f)$ is Zoll
for some sequence $\lambda_n \downarrow 0$.

\begin{prop}
Let $(g,f)$ be a magnetic system on $\T^2$ such that 
 $(g,\lambda_n f)$ is Zoll for some sequence $\lambda_n \downarrow 0$. Then $g$ is a flat metric.
\label{prop:lambdasmall}
\end{prop}

\begin{rmk}
The proof of Proposition \ref{prop:lambdasmall} actually shows that there exists $\lambda_-=\lambda_-(g,f)>0$ such that
if $(g,\lambda_0 f)$ is Zoll for some $\lambda_0 \in (0,\lambda_-)$ then $g$ is a flat metric. It would be interesting to see whether
$\lambda_-$ can be chosen independently of the magnetic systems in a neighborhood of $(g,f)$, 
and more generally if after some normalization (such as e.g. Area$(\T^2,g)=1$, $\|f\|_\infty\leq 1$, ...) 
the same assertion holds from some constant $\lambda_-$ independent of the magnetic system.
\end{rmk}

Proposition \ref{prop:lambdasmall} is an immediate corollary of Proposition \ref{prop:persistencetorus} below on the persistence of closed geodesics under magnetic perturbations, whose statement requires the introduction of some notation (for the details we refer e.g. to \cite{Abbondandolo:2013is,Contreras:2006yo}). We recall that closed arc-length parametrized geodesics on $(\T^2,g)$ one-to-one correspond to the critical points of the free-period Lagrangian action functional 
\begin{equation}
\A:H^1(\T,\T^2)\times (0,+\infty)\rightarrow \R, \quad \A(\Gamma,\tau) := \frac 1{2\tau} \int_0^1 |\dot \Gamma(s)|^2 \, \diff s + \frac{\tau}{2},
\label{freeperiod1}
\end{equation}
meaning that $\gamma:\R/T\Z\to S^2$ is an arc-length parametrized closed geodesic if and only if $(\Gamma,T)$ is a critical point of $\A$, where $\Gamma$ is given 
by $\Gamma(s) := \gamma(Ts)$.  
Here $H^1(\T,\T^2)$ denotes the space of one-periodic loops in $\T^2$ of Sobolev-class $H^1$, and it is well-known that its connected 
components are in bijection with elements (actually conjugacy classes) of $\pi_1(\T^2)$. Hereafter we will identify a pair $(\Gamma,\tau)$ with the corresponding curve 
$\gamma$, and write $\A(\gamma)$ instead of $\A(\Gamma,\tau)$ whenever more convenient.
An analogous variational principle is available also for exact magnetic systems (i.e. when $f\mu_g$ is exact) and allows us to detect closed $(g,\lambda f)$-geodesics 
as critical points of a suitable action functional $\A^\lambda$, whose precise definition will be recalled in \eqref{freeperiod2}.

For every homotopy class $\alpha\in \pi_1(\T^2)\setminus \{0\}$ we denote by $K_\alpha\neq \emptyset$ the compact set of global minimizers of $\A$ 
in the connected component of $H^1(\T,\T^2)\times (0,+\infty)$ determined by $\alpha$, that is 
$$\A(\gamma_\alpha) = \min_{\gamma\in \alpha} \A(\gamma) \quad \text{iff}\quad \gamma_\alpha \in K_\alpha.$$
We would like to stress that in general the set $K_\alpha$ does not have more structure than a compact set (e.g. it could be a Cantor set). 

\begin{prop}
Suppose that $g$ is not a flat metric on $\T^2$, and let $f:\T^2\to \R$ be any smooth function. Then there exist a
homotopy class $\alpha\in \pi_1(\T^2)\setminus \{0\}$, a bounded neighborhood $\U_\alpha$ of the set $K_\alpha$,
and a constant $\Lambda =\Lambda (g,f,\U_\alpha)>0$ such that for every $0<\lambda <\Lambda$ there exists a 
closed $(g,\lambda f)$-geodesic which is contained in $\U_\alpha$ and is a local minimizer of the functional $\A^\lambda$ in 
\eqref{freeperiod2}.
\label{prop:persistencetorus}
\end{prop}

\begin{proof}[Proof of Proposition \ref{prop:lambdasmall}]
Suppose that $g$ is not flat and set $\lambda_-=\Lambda (g,f,\U_\alpha)$, where $\alpha\in \pi_1(\T^2)\setminus \{0\}$ is given by Proposition \ref{prop:persistencetorus}. 
For $n\in \N$ such that $\lambda_n\in (0,\lambda_-)$ we thus find a closed non-contractible $(g,\lambda_n f)$-geodesic, in contradiction with the fact that all $(g,\lambda_n f)$-geodesics must be contractible, see Remark \ref{rmk:allclosed}.
\end{proof}

Before proceeding with the proof, we would like to make two comments on Proposition \ref{prop:persistencetorus}. First, to establish if closed 
geodesics are stable under magnetic perturbations is a very natural question which has been already investigated in the past decades. 
Following \cite{Ginzburg:1994} we see that, while on the one hand a non-degenerate closed geodesic always ``survives'' when switching on a magnetic field, 
on the other hand the example of a flat torus with induced area form shows that we must impose some 
kind of condition on the closed geodesic for it not to disappear. Therefore, Proposition \ref{prop:persistencetorus}
can be seen as a first step towards the study of the stability of degenerate closed geodesics under magnetic perturbations. Another instance 
of this persistence will be discussed in Section \ref{section:4}.

Second, we would like to stress that in Proposition \ref{prop:persistencetorus} we do not 
require the magnetic function $f$ to have vanishing integral over $\T^2$, or, equivalently,
the two-form $f\mu_g$ to be exact. Hence, in general, for the Ma\~n\'e critical value of the universal cover defined in \eqref{mane} we have
$$c(g,0) = 0, \quad c(g,\lambda f)=+\infty \ \ \forall \ \lambda >0,$$
or, equivalently, for the constant $h(g,f)$ defined in \eqref{oneovermane}
$$h(g,0)=+ \infty, \quad h(g,\lambda f) = 0 \ \ \forall \ \lambda >0.$$
This can be rephrased by saying that the magnetic perturbation is \textit{not} small even if $\lambda$ is (arbitrary) small.

\begin{proof}[Proof of Proposition \ref{prop:persistencetorus}]
We first assume that the image of $K_\alpha$ under the evaluation map 
$$\text{ev} : H^1(\T,\T^2 ) \times (0,+\infty) \times \T \to \T^2, \quad \text{ev} (\gamma,s) := \gamma(\tau s),$$ 
is a proper compact subset of $\T^2$ for some $\alpha \in \pi_1(\T^2)\setminus \{0\}$. Clearly, under this assumption we can find a bounded neighborhood 
$\U_\alpha\subseteq H^1(\T,\T^2)\times (0,+\infty)$ of $K_\alpha$ such that 
$$U:=\text{ev}(\U_\alpha\times \T)\subset \text{ev}(\overline {\U}_\alpha \times \T) \subsetneq \T^2.$$ 
Thus, we have $f \mu_g|_U=\diff \theta$ for some bounded one-form $\theta\in \Omega^1(U)$. According to \cite{Contreras:2006yo} 
closed $(g,\lambda f)$-geodesics with image contained in $U$ correspond to critical points of the free-period Lagrangian action functional
\begin{equation}
\A^{\lambda} (\Gamma,\tau ) := \frac{1}{2\tau}\int_0^1 |\dot \Gamma(s)|^2 \diff s - \lambda \int_0^1 \theta_\Gamma(\dot \Gamma)\, \diff s + \frac \tau 2.
\label{freeperiod2}
\end{equation}
At the same time, since $\A$ satisfies the Palais-Smale condition on the connected components of $H^1(\T,\T^2)\times (0,+\infty)$, there exists $\epsilon >0$ such that 
$$\inf_{\partial \U_\alpha} \A > \A(K_\alpha)+\epsilon,$$
where $\A(K_\alpha)$ denotes the action of any element in $K_\alpha$ (see e.g. \cite[Lemma 3.1]{Abbondandolo:2014rb} for the proof). We now 
show that, if $\lambda>0$ is sufficiently small, then 
\begin{equation}
\inf_{\partial \U_\alpha} \A^\lambda  \geq \sup_{\gamma_\alpha \in K_\alpha} \A^\lambda(\gamma_\alpha)+\frac\epsilon2.
\label{equationlambdasmall}
\end{equation}
From the fact that $\A^\lambda$ satisfies the Palais-Smale condition on bounded subsets of $H^1(\T,S^2)\times (0,+\infty)$ we therefore 
deduce that for such values of $\lambda$ there exists a closed $(g,\lambda f)$-geodesic which is contained in $\U_\alpha$ and is a global minimizer of 
$\A^\lambda$ in $\U_\alpha$, thus completing the proof.

To prove \eqref{equationlambdasmall} we preliminarly compute using $\theta_q(v)=\langle X_q,v\rangle$ and the Cauchy-Schwarz inequality
$$\left |\int_0^1 \theta_\Gamma(\dot \Gamma)\, \diff s \right |\leq \int_0^1 |\theta_\Gamma(\dot \Gamma)|\, \diff s = \int_0^1 |\langle X_\Gamma,\dot \Gamma\rangle |\, \diff s\leq \|X_\Gamma\|_2 \|\dot \Gamma\|_2 \leq \|\theta\|_\infty  \|\dot \Gamma\|_2$$
and set 
$$r:=\sup_{(\Gamma,\tau)\in \overline{\U}_\alpha} \|\dot \Gamma\|_2, \quad \Lambda=\Lambda(g,f,\U_\alpha) := \frac{\epsilon}{4r\|\theta\|_\infty}.$$ 
For all $\lambda < \Lambda$, all $\gamma=(\Gamma,\tau)\in \partial \U_\alpha$, and all $\gamma_\alpha =(\Gamma_\alpha, \tau_\alpha) \in K_\alpha$ we thus have
\begin{align*}
\A^\lambda (\gamma) & = \A ( \gamma) - \lambda \int_0^1 \theta_\Gamma(\dot \Gamma)\, \diff s \\ 
				  &\geq \A(\gamma) - \lambda \|\theta\|_\infty\|\dot \Gamma\|_2 \\
				  &\geq \A(\gamma) - \frac \epsilon 4 \\
				  &> \A(\gamma_\alpha) + \frac 34 \epsilon \\
				  &\geq \A(\gamma_\alpha) + \lambda \|\theta\|_\infty\|\dot \Gamma_\alpha \|_2 + \frac \epsilon 2 \\
				  &\geq \A^\lambda (\gamma_\alpha) + \frac \epsilon 2,
\end{align*}
and \eqref{equationlambdasmall} follows taking the infimum over $\gamma\in \partial \U_\alpha$ and the supremum over $\gamma_\alpha \in K_\alpha$.

Thus, we are left to consider the case in which each set $K_\alpha$ is mapped surjectively onto $\T^2$ by the evaluation map ev. In this case we will show 
that the metric $g$ must be flat, in contradiction with the assumption. 

We preliminarly observe that the same proof as above allows us to find closed $(g,\lambda f)$-geodesics close to any ``isolated'' $\gamma_\alpha\in K_\alpha$. Here by isolated we mean that $\gamma_\alpha$ is a strict local (actually global) minimizer of $\A$, that is, 
there exists a neighborhood $\U$ of the critical circle $\T\cdot \gamma_\alpha :=\{\gamma_\alpha(\cdot+s)\, |\, s\in \T\}$ such that
$$\inf_\U \A = \A(\gamma_\alpha), \quad \text{and}\ \ \U\cap \A^{-1}(\A(\gamma_\alpha)) = \T\cdot \gamma_\alpha.$$
Therefore, we can further assume that all $\gamma_\alpha\in K_\alpha$ are non-isolated. We claim that, under this assumption, for any $\alpha\in \pi_1(\T^2)\setminus \{0\}$ 
the set $K_\alpha$ yields a simple foliation of $\T^2$ by closed geodesics. Observe that this immediately implies 
that the metric is flat by a theorem of Innami \cite{Innami:1986} (see also \cite{Bangert:1994}). 

To prove the last assertion we show that two elements $\gamma_\alpha,\nu_\alpha\in K_\alpha$ that are not the same geometric curve
must have disjoint image. Thus, let us suppose that $\nu_\alpha,\gamma_\alpha\in K_\alpha$ intersect transversally, and denote with $\rho_\alpha$, $\tau_\alpha$ their periods.
Then, we can find lifts $\tilde \nu_\alpha,\tilde \gamma_\alpha:\R\to \R^2$ of $\nu_\alpha$, $\gamma_\alpha$ to $\R^2$ respectively such that 
$$ \tilde \nu_\alpha |_{[0,\rho_\alpha]} (\cdot) \ \cap\ \tilde \gamma_\alpha|_{[0,\tau_\alpha]}(\cdot) \neq \emptyset.$$
Observe that $\tilde \nu_\alpha$, $\tilde \gamma_\alpha$ are embedded. Also, since $\nu_\alpha$ and $\gamma_\alpha$ belong to the same homotopy class, 
up to shifting the base point of $\tilde \nu_\alpha$ we can suppose that 
there exist $0\leq s_1<s_2<\rho_\alpha$ and $0\leq t_1<t_2<\tau_\alpha$ such that 
$$\tilde \nu_\alpha(s_i)= \tilde \gamma_\alpha(t_i), \quad \text{for} \ i=1,2.$$
We now define the piecewise smooth curves 
\begin{align*}
\tilde \eta_1 &:= \tilde \nu_\alpha |_{[0,s_1]} \,\# \, \tilde \gamma_\alpha |_{[t_1,t_2]} \, \# \, \tilde \nu_{\alpha}|_{[s_2,\rho_\alpha]},\\ 
\tilde \eta_2 &:= \tilde \gamma_\alpha |_{[0,t_1]} \, \# \, \tilde \nu_{\alpha} |_{[s_1,s_2]}\, \#\, \tilde \gamma_\alpha |_{[t_2,\tau_\alpha]}.
\end{align*}
By construction, $\tilde \eta_1$ and $\tilde \eta_2$ project to closed curves $\eta_1$ and $\eta_2$ on $\T^2$ in the homotopy 
class $\alpha$ that satisfy
$$\A(\eta_1) + \A(\eta_2) = \A(\nu_\alpha) + \A(\gamma_\alpha) = 2\A(\gamma_\alpha).$$
It follows that at least one of them, say $\eta_1$, has action less or equal to $\A(\gamma_\alpha)$. Therefore, $\eta_1$ is a closed geodesic (for it is a global minimizer of $\A$ in the homotopy class $\alpha$), hence in particular smooth, a contradiction. 

Using this we readily see that a non-isolated $\gamma_\alpha$ must be:
\begin{enumerate}
\item embedded if $\alpha$ is a primitive class in $\pi_1(\T^2)\setminus \{0\}$, or
\item the $m$-th iterate of some $\gamma_\beta \in K_\beta$, if $\alpha =m\cdot \beta$ for some $n\in \N$ and some primitive class $\beta\in \pi_1(\T^2)$.
\end{enumerate}

Indeed, if $\gamma_\alpha$ had a transversal self-intersection, then the image of elements in $K_\alpha$ sufficiently close to $\gamma_\alpha$
would not be disjoint from the image of $\gamma_\alpha$. It is now straightforward to see that under our assumptions the set $K_\alpha$ yields a simple 
foliation of $\T^2$ by geodesics if $\alpha$ is a primitive class in $\pi_1(\T^2)$ and an $m$-fold foliation by geodesics for some $m\in \N$ otherwise.
This completes the proof.
\end{proof}

\section{Zoll magnetic systems on flat tori for large values of $\lambda$}
\label{section:3}

Let $(g,f)$ be a magnetic system on $\Sigma=\T^2$ as in the statement of Theorem \ref{thm:main}. In virtue of Proposition \ref{prop:lambdasmall} we can assume 
that $g=g_\fl$ is a flat metric. In this section, by looking at large values of $\lambda$, we show that the magnetic function $f$ must be constant as well, 
thus completing the proof of Theorem \ref{thm:main} in the case $\Sigma=\T^2$. 

The theorem is an immediate corollary of Proposition \ref{prop:lambdasmall} combined with the following 

\begin{prop}
Let $(g_\fl,f)$ be a magnetic sytem on $\T^2$ such that $(g,\lambda_n f)$ is Zoll for some sequence $\lambda_n\uparrow +\infty$. 
Then $f$ is constant. 
\label{prop:fconstant}
\end{prop}

\begin{rmk}
As for Proposition \ref{prop:lambdasmall}, we can improve the statement of Proposition \ref{prop:fconstant} by saying that there exists $\lambda_+=\lambda_+(g,f)>0$ such that, if 
$(g_\fl,\lambda_1 f)$ is Zoll for some $\lambda_1\in (\lambda_+,+\infty)$, then $f$ must be a constant function. Also here it would be 
very interesting to see if - after normalization - the constant $\lambda_+$ can be chosen independently of the magnetic system.
\end{rmk}

The idea of the proof is that for $\lambda>0$ large enough we always find short periodic $(g_\fl,\lambda f)$-geodesics, 
and if $f$ is not constant also long ones (close to a regular level of $f$). 
The fact that $(g_\fl,\lambda f)$ is Zoll then contradicts the dichotomy between long and short periodic orbits established in \cite[Theorem 7.13]{Benedetti:diss}. 
We shall notice that \cite[Theorem 7.13]{Benedetti:diss} deals only with magnetic systems on the two-sphere; however, the proof is based on an argument of
Bangert \cite{Bangert:1986} which works for flows converging to a free $S^1$-action, and hence extends e.g. to all magnetic systems on closed surfaces.

\begin{thm*}{\cite[Theorem 7.13]{Benedetti:diss}}
Let $(g,f)$ be a magnetic system on $\Sigma$, with $f>0$. For every $\epsilon>0$ and $n\in\N$ there is $\Lambda=\Lambda(\epsilon,n)\in (0,+\infty)$ such that 
for every $\lambda \in (\Lambda,+\infty)$ a periodic prime $(g,\lambda f)$-geodesic is either a simple curve with length in $(\frac{2\pi-\epsilon}{\lambda \max f}, \frac{2\pi +\epsilon}{\lambda \min f})$, or has at least $n$ self-intersections and length larger than $\frac{1}{\lambda \epsilon}$.
\end{thm*}

In order to apply this theorem we need some preliminary results on the properties of magnetic trajectories in $\R^2$ for sufficiently small energies. Such properties 
are certainly well-known to the experts (see the discussion at the end of the section), however for the reader's convenience we include them here.

\begin{lem}\label{lem:magnetic_localization} Let $(g, f)$ be a magnetic pair on $\Sigma$. For any neighborhood $U$ of any compact subset $K \subset \Sigma$ with $f|_U> 0$, any $T>0$ there exists some $\Lambda>0$ such that for every $\lambda > \Lambda$ every $(g,\lambda f)$-geodesic $\gamma$ starting in $K$ stays completely in $U$ in forward 
and backward time up to time $T$, that is, satisfies $\gamma|_{[-T,T]}\subset U$.
\end{lem}
\begin{proof} The claim follows immediately from the fact that over the interior of $\mathrm{supp}(f)$ the magnetic flow $C_{\mathrm{loc}}^{\infty}$-converges to rotations in the fibers;
see \cite{Kerman:1999}.
%We will see that the compactness assumption reduces the claim to a local statement and therefore for simplicity assume that $\Sigma=\R^2$ right away. For $z\in K$ and $\mu\geq 0$ we define 
%$$f_{\mu,z}(q):=f(\mu q+z) ),\quad (g_{\mu,z})_q(v,w):=g_{\mu q+z}(v,w).$$ 
%The pair $(g_{\mu,z},f_{\mu,z})$ depends continuously on $\mu$ and $z$, and $C_{\mathrm{loc}}^{\infty}$-converges to $(g_z,q\mapsto f(z)$ for $ \mu \rightarrow 0$. After rescaling, the lemma becomes an implication of the following statement: For any sufficiently small $\mu > 0$, any $z\in K$ and any $T>0$, there is some %$\lambda=\lambda(\mu,z,T)>1/\mu$ such that $(g_{\mu,z},\lambda \mu f_{\mu,z})$-geodesics starting at $0$ stay in the unit ball $B_1(0)$
%%in forward and backward time up to time $T$. This happens if and only if for any sufficiently small $\mu > 0$, any $z\in K$ and any $T>0$ there is some
%$\lambda=\lambda(\mu,z,T)>1$ such that $(g_{\mu,z},\lambda f_{\mu,z})$-geodesics starting at $0$ stay in the unit ball $B_1(0)$ in forward and backward time up to time $T$. Since solutions of our equation of motion depend continuously on $z$ and $\mu$, it suffices to show that there exists some $\lambda$ independent of $z$ and $T$ for which the latter claim holds for $\mu=0$ at $z \in K$ (by compactness). This is indeed the case, as for $\mu=0$ solutions are ellipses which converge to a point for $\lambda \rightarrow\infty$.
\end{proof}

\begin{lem}\label{lem:long_geodesic} Let $(g_\fl, f)$ be a magnetic pair on $\R^2$ and let $q\in \R^2$ be a regular point of $f$. Then for sufficiently large $\lambda$ there is a $(g_\fl,\lambda f)$-geodesic starting at $q$ which has self-intersections.
\end{lem}
\begin{proof} Suppose that the level set $\{f=e\}\ni q$ is tangential to the $x$-axis at $q$ and that the gradient of $f$ at $q$ points in the positive $y$-direction. We look at $(g_\fl,\lambda f)$-geodesics starting at $q$ in the negative $y$-direction. For sufficiently large $\lambda$ the trajectory $q(t)$ schematically looks as the solid curve in Figure \ref{fig:magnetic_path}. For given $\varepsilon>0$ we choose $\delta>0$ such that 
\[
			\|f(q+v)-f(q)-D_qf \cdot v\|\leq \varepsilon \|v\|
\]
for all $v\in \R^2$ with $\|v\|\leq \delta$. Let $t_1$ and $t_2$ be the first and second time at which $q(t)$ hits the $x$-axis. By the proof of Lemma \ref{lem:magnetic_localization} there exists some $c>1$ such that for all sufficiently large $\lambda$ the segment $q|_{[0,t_2]}$ is contained in a ball with radius $r_{\lambda}=c/(e\lambda)$ around $q$. We can suppose that this ball is contained in $B_q(\delta)$ by choosing $\lambda$ large enough. In order to prove that $q(t)$ is not embedded, it suffices to show that the ``drift'' $\Delta x$ depicted in Figure \ref{fig:magnetic_path} is positive for $\lambda$ large enough. 

\begin{figure}[h]
\begin{small}
	\centering
		\def\svgwidth{0.6\textwidth}
		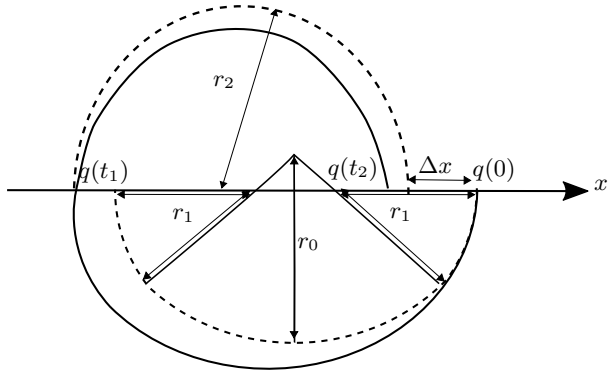
	\caption{Magnetic trajectory in a magnetic field on $\R^2$.}
	\label{fig:magnetic_path}
	\end{small}
\end{figure}

A curvature comparison shows that the segments $q|_{[0,t_1]}$ and $q|_{[t_1,t_2]}$ of $q(t)$ lie completely below the dashed lines in Figure \ref{fig:magnetic_path} which are specified as follows: the dashed line above the $x$-axis is a circular arc with radius 
$$r_2:=\frac{1}{\lambda e-\varepsilon r_{\lambda}}$$ 
starting at $q(t_1)$ tangentially to $q(t)$. The dashed line below the $x$-axis is tangential to $q(t)$ at $q(0)$ and consists of three circular arcs of radii 
$$r_1:=\frac{1}{\lambda e+\varepsilon r_{\lambda}}, \quad r_0:=\frac{1}{\lambda e-\frac{L\lambda}{2(\lambda e+\varepsilon r_{\lambda})}+\varepsilon r_{\lambda}},\quad \text{and} \ \ r_1$$
respectively, which meet tangentially on the horizontal line specified by $y=2(\lambda e+\varepsilon r_{\lambda})$. Here we have set $L=\|D_qf\|$. We compute

    \begin{equation}
    \begin{split}
   \Delta x &\geq 2r_1 +2\cos(30^{\circ})(r_0-r_1)-2r_2  \\
	 &\geq \frac{2 \cos(30^{\circ})}{\lambda e-\frac{L\lambda}{2(\lambda e+\varepsilon c/(e\lambda))}+\frac{c\varepsilon}{e\lambda}}+ \frac{2(1-\cos(30^{\circ}))}{\lambda e+\frac{c\varepsilon}{e\lambda}}-\frac{2}{\lambda e-\frac{c\varepsilon}{e\lambda}} \\
	 &:= 2\Delta(\varepsilon,\lambda)
    \end{split}
    \end{equation}
for some constant $c>0$. We need to show that there exists some $\varepsilon>0$ and a sufficiently large $\lambda$ such that $\Delta(\varepsilon,\lambda)>0$. Since $|\partial_{\varepsilon} \lambda^2\Delta(\varepsilon,\lambda)|$ is bounded from above on $(\varepsilon,\lambda)\in (0,1]\times [1, \infty)$, it suffices to show that $\liminf_{\lambda\rightarrow \infty} \lambda^2\Delta(0,\lambda)>0$. This is indeed the case, for
  \begin{equation*}
   \lambda^2\Delta(0,\lambda) = \frac{\cos(30^{\circ})}{\lambda e-\frac{L\lambda}{2\lambda e}}+ \frac{1-\cos(30^{\circ})}{\lambda e}-\frac{1}{\lambda e} = \frac{L\lambda \cos(30^{\circ})}{2e^2(\lambda e-\frac{L}{2e})}. \qedhere
  \end{equation*}
\end{proof}

%Moreover, we need the following statement.
%\begin{lem}
%Let $(g_\fl,f)$ be a magnetic sytem on $\T^2$. If $(g_\fl,\lambda f)$ is Zoll for arbitrarily large $\lambda$, then $f$ satisfies $f\geq 0$ (or  $f\leq 0$).
%\label{lem:signf}
%\end{lem}
%\begin{proof} By \cite[Proposition 1.3]{Benedetti:2018c} we can assume that $\int_{\T^2} f \mu_{g_\fl} > 0$. If $f\geq 0$ does not hold, then $f \mu_{g_\fl}$ is oscillatory in the sense of \cite{Asselle:2015hc}. In particular, by \cite[Theorem 4.1]{Asselle:2015hc} for sufficiently large $\lambda$ there exists a periodic $(g_\fl,\lambda f)$-geodesic $\gamma$ which is a local minimizer of the action functional $\A^{\lambda}$. Recall that in the Zoll case the functional $\A^{\lambda}$ is constant on the set of all prime magnetic geodesics. Another periodic magnetic geodesic $\nu$ starting at $\gamma(0)$ with intial velocity close to $\dot{\gamma}(0)$ has to intersect $\gamma$ a second time. Now the same line of arguments as in the end of the proof of Proposition \ref{prop:lambdasmall} applied to $\gamma$ and $\nu$ and using the functional $\A^{\lambda}$ instead of $\A$ yields a contradiction.
%\end{proof}

Before proving Proposition \ref{prop:fconstant} we shall recall that on $\mathbb{T}^2$ the action functional $\A^\lambda$
is well-defined over the space of contractible loops even if $f\mu_g$ is not exact. Indeed, one can replace the integral of $\theta$ along a contractible loop $\gamma$ 
by the integral of $f\mu_g$ over a capping disk for $\gamma$, and the resulting functional will not depend on the choice of 
the capping disk since $\T^2$ is aspherical. Moreover, for any Zoll magnetic system $(g,f)$ the value of the functional $\A^\lambda$ 
is constant on the set of prime closed $(g,f)$-geodesics; this latter fact follows e.g. from the magnetic systolic-diastolic inequality in \cite{Benedetti:2018c}.

\begin{proof}[Proof of Proposition \ref{prop:fconstant}]

We claim that $f\geq 0$. Indeed, first by \cite[Proposition 1.3]{Benedetti:2018c} we can assume that 
the integral of $f$ over $\T^2$ is positive. Now, if $f\geq 0$ does not hold, then $f \mu_{g_\fl}$ is oscillating in the sense of \cite{Asselle:2015hc}. In particular, by \cite[Theorem 4.1]{Asselle:2015hc} for sufficiently large $\lambda$ there exists a periodic $(g_\fl,\lambda f)$-geodesic $\gamma$ which is a local minimizer of the action functional $\A^{\lambda}$. Another periodic magnetic geodesic $\nu$ starting at $\gamma(0)$ with initial velocity close to $\dot{\gamma}(0)$ has to intersect $\gamma$ a second time. Now, since in the Zoll case the functional $\A^{\lambda}$ is constant on the set of all prime magnetic geodesics, the same line of arguments as in the end of the proof of Proposition \ref{prop:persistencetorus} applied to $\gamma$ and $\nu$ and using the functional $\A^{\lambda}$ instead of $\A$ yields a contradiction.

%In case $\Sigma$ is not a two-torus, an alternative argument is the following: if $f\mu_g$ is exact or if $\int_\Sigma f>0$ and $f<0$ somewhere, the energy level 
%$(T^*_1\Sigma, \diff p \wedge \diff q + \lambda \pi^* (f\mu_g))$ corresponding to the pair $(g,\lambda f)$ cannot be
%of contact type for $\lambda >0$ sufficiently large (see \cite{Contreras:2004lv,Asselle:2015ij}), in contradiction with the fact that the energy level of a Zoll pair must be of 
%contact type (see \cite{Benedetti:2018c}). 
We first consider the case $f>0$. We choose $\epsilon>0$ so small that 
\begin{equation}
\left [\frac{2\pi -\epsilon}{ \max f}, \frac{2\pi + \epsilon}{\min f}\right ]\cap \left [\frac1\epsilon,+\infty\right ) =\emptyset,
\label{epsilondisjoint}
\end{equation}
and set $\Lambda := \Lambda(\epsilon,1)>0$ as in \cite[Theorem 7.13]{Benedetti:diss}.
Following \cite{Ginzburg:1987lq}, up to enlarging $\Lambda$ if necessary we find for all $\lambda>\Lambda$
 one\footnote{Actually, at least two.} embedded closed $(g_\fl,\lambda f)$-geodesic $\gamma_\sss^\lambda$ which is short, meaning that its length satisfies 
$$\frac{2\pi -\epsilon}{\lambda \max f}\ < \ \ell (\gamma_\sss^\lambda) \ < \  \frac{2\pi + \epsilon}{\lambda \min f},\quad \forall \lambda >\Lambda.$$ 

On the other hand, if $f$ is not constant, then by Lemma \ref{lem:long_geodesic} up to enlarging $\Lambda$ further 
we also find for all $\lambda>\Lambda$ a $(g_\fl,\lambda f)$-geodesic $\gamma_\llll^\lambda$ which has self-intersections. We set $\lambda_+:=\Lambda$ and 
take $n\in \N$ such that $\lambda_n>\lambda_+$. Then, the $(g_\fl,\lambda_n f)$-geodesic $\gamma_\llll^{\lambda_n}$ 
with self-intersections is necessarily closed and hence by \cite[Theorem 7.13]{Benedetti:diss} long, meaning that   
 $$\ell (\gamma_\llll^{\lambda_n} ) > \frac{1}{\lambda_n\epsilon}.$$

We denote by $(q_\sss,v_\sss)$, $(q_\llll,v_\llll)$ the initial conditions of the closed $(g_\fl,\lambda_n f)$-geodesics 
$\gamma_\sss^{\lambda_n},\gamma_\llll^{\lambda_n}$ respectively, and consider a path $r\mapsto (q(r),v(r))$ 
in $S\T^2$  connecting $(q_\sss,v_\sss)$ to $(q_\llll,v_\llll)$. Since the period (and hence the length) of closed $(g,\lambda_n f)$-geodesics 
varies smoothly, by the intermediate value theorem we find $r_0\in (0,1)$ such that the length of the closed $(g,\lambda_n f)$-geodesic 
$\gamma_\iii^{\lambda_n}$ with initial conditions $(q(r_0),v(r_0))$ satisfies 
$$\frac{2\pi +\epsilon}{\lambda_n \min f}\ < \ \ell(\gamma_\iii^{\lambda_n}) \ <  \ \frac{1}{\lambda_n \epsilon},$$
thus contradicting \cite[Theorem 7.13]{Benedetti:diss}.

The case $f\geq 0$ is more delicate since \cite[Theorem 7.13]{Benedetti:diss} cannot be applied directly. 
Our strategy is to show that the problem can be reduced
to the case $f>0$ by suitably modifying the magnetic function. Suppose that $f\geq 0$ is not constant and fix three regular energy values $0<e_0<e_1<e_2<\max f$. Without loss of generality 
we can suppose that the superlevel sets $\{f\geq e_i\}$ are connected, $i=0,1,2$. Using a cut-off function we construct $\tilde f$ such that 
$$\tilde f \equiv f\quad \text{on}\ \{f\geq e_1\}, \quad \text{and} \ \tilde f\equiv e_0\quad \text{on}\ \{f\leq e_0\}.$$
%We choose $\Lambda(\varepsilon,1)$ according to \cite[Theorem 7.13]{Benedetti:diss} so that the interval of ``short'' and ``long'' closed $(g,\tilde \lambda f)$-geodesics is disjoint. Let $\overline T$ be an ``intermediate length''.
We fix $\epsilon >0$ such that \eqref{epsilondisjoint} holds with $f$ replaced by $\tilde f$ and choose 
$$\bar \ell\in \Big (\frac{2\pi+\epsilon}{e_0}, \frac{1}{\epsilon}\Big).$$ 
We also let $\Lambda >0$ be such that, for all $\lambda>\Lambda$, closed $(g_\fl,\lambda \tilde f)$-geodesics are either embedded 
and short (in the sense above) or long and have self-intersections.  Up to enlarging $\Lambda$ we find for all $\lambda>\Lambda$ at least two short embedded closed $(g_\fl,\lambda \tilde f)$-geodesics, one of which is also a short $(g_\fl, \lambda f)$-geodesic (namely, the one that is close to the $S^1$-fiber over a maximum point of $\tilde f$). We denote such 
a short periodic orbit with $\gamma_\sss^\lambda$ and recall that
$$\ell(\gamma_\sss^\lambda) \in \Big (\frac{2\pi -\epsilon}{\lambda \max f}, \frac{2\pi + \epsilon}{\lambda e_0}\Big ),\quad \forall \lambda >\Lambda.$$

At the same time, by Lemmas \ref{lem:magnetic_localization} and \ref{lem:long_geodesic} applied to a neighborhood $U\subset \{f> e_1\}$ of the 
compact set $K=\{f\geq e_2\}$, up to enlarging $\Lambda$ further we can assume that for all $\lambda>\Lambda$ there is a 
$(g_\fl,\lambda f)$-geodesic $\gamma_\llll^\lambda$ starting at $f^{-1}(e_2)$ which does not close up until time $\bar \ell/\lambda$, and that all $(g_\fl,\lambda f)$-geodesics starting in $K$ stay in $U$ up to time $\bar \ell/\Lambda$. %which is not embedded and whose length is larger than $\bar \ell$. ``long'' with respect to $(g,\tilde \lambda f)$. In addition, by Lemma \ref{lem:magnetic_localization} every $(g,f)$-geodesic starting in $K$ will stay in $U$ up to time $\overline{T}$
We set $\lambda_+:=\Lambda$ and take $n\in \N$ such that $\lambda_n>\lambda_+$. In particular, the $(g_\fl,\lambda_n f)$-geodesic 
$\gamma_\llll^{\lambda_n}$ is closed and has length larger than $\bar \ell/\lambda_n$.

Let now $(q_\sss,v_\sss)$, $(q_\llll,v_\llll)$ be the initial conditions of $\gamma_\sss^{\lambda_n}$ and $\gamma_\llll^{\lambda_n}$, respectively, and let
$r\mapsto (q(r),v(r))$ be a path in $S\Sigma$ from $(q_\sss,v_\sss)$ to $(q_\llll,v_\llll)$
such that $q(\cdot)$ is entirely contained in $\{f\geq e_2\}$. Since the period of closed 
$(g,\lambda_n f)$-geodesics varies smoothly, by the intermediate value theorem we find $r_0\in (0,1)$ such that the 
corresponding closed $(g,\lambda_n f)$-geodesic $\gamma_\iii^{\lambda_n}$ has length $\bar \ell/\lambda_n$, and in particular 
is completely contained in $\{f\geq e_1\}$. This shows that $\gamma_\iii^{\lambda_n}$ is also a closed $(g,\lambda_n \tilde f)$-geodesic, thus contradicting \cite[Theorem 7.13]{Benedetti:diss}.
\end{proof}

Lemma \ref{lem:long_geodesic} can also be deduced from Corollary 1.4 in \cite{Raymond:2015}, in which 
 low energy magnetic dynamics in $\R^2$ under a non-constant magnetic field is described in terms of a fast rotating motion and a slow drift along the level sets
of the magnetic function. Let us mention that Proposition \ref{prop:fconstant} extends with the same proof to magnetic systems $(g,f)$ where $g$ is any metric of constant curvature. For an arbitrary  system $(g,f)$ an analogue of Proposition \ref{prop:fconstant}, and therewith a step towards Conjecture (Z), would require to unscramble the influences of an inhomogeneous magnetic field and an inhomogeneous metric on the drift motion. We intend to investigate this in future work.

%This description is obtained from the study of the normal form of the Hamiltonian systems for low energies. Since such a normal form is available on 
%any closed Riemannian surface (see \cite{Castilho:2001}), we believe that Lemma \ref{lem:long_geodesic} can be extended to arbitrary magnetic systems. 
%As a consequence, this would extend the validity of Proposition \ref{prop:fconstant} to any magnetic system on any 
%closed surface. This problem will not be addressed here, since on the one hand we do not need an analogous of Proposition \ref{prop:fconstant}
%to prove Theorem \ref{thm:main} for higher genus surfaces, and on the other hand we are still at a very preliminary stage in the proof of Conjecture (Z) for $\Sigma=S^2$;
%see Section \ref{section:4}.
%Here we only mention that Proposition \ref{prop:fconstant} extends with the same proof to magnetic systems $(g,f)$ where $g$ is any metric of constant curvature.
 %Also, preliminary investigations suggest that a similar argument should allow to prove Proposition \ref{prop:fconstant} for all 
%magnetic systems $(g,f)$ for which the critical set of the Gaussian curvature is disjoint from the critical set of $f$.

We now prove Theorem \ref{thm:main} for higher genus surface. Actually, we will prove the following stronger 

\begin{thm}
Let $(g,f)$ be a magnetic system on a surface with genus at least two such that $(g,\lambda_n)$ is Zoll for some sequence $\lambda_n \downarrow h(g,f)$. 
Then $g$ has constant curvature and $f$ is a constant function.
\label{thm:highergenus}
\end{thm}

\begin{proof}
Recall that the \textit{helicity} $\mathcal H(g,\lambda f)$ of $(g,\lambda f)$ vanishes if and only if $\lambda =\lambda_{g,f}$, where
$$\lambda_{g,f}^2 = \frac{-2\pi \chi (\Sigma) A}{[f]^2}.$$
Here $A$ is the total area of $(\Sigma,g)$, and $[f]:=\int_\Sigma f\mu_g$ is the total integral of $f\mu_g$.
For our purposes we don't need to recall the definition of helicity; all we need to know is
that, by Lemma 2.12 in \cite{Benedetti:2018c} the helicity $\mathcal H(g,\lambda f)$ is related to the average magnetic curvature 
$K_{\lambda f}:= \lambda^2[f]^2 +\ 2\pi \chi(\Sigma)A^{-1}$
by the formula
$$\mathcal H(g,\lambda f) = \frac{A^2}{2\chi(\Sigma)} \, K_{\lambda f},$$
hence in particular the function $\lambda\mapsto \mathcal H(g,\lambda f)$ is strictly monotonically decreasing. Moreover, 
$\lambda_{g,f}$ and $h(g,f)$ are related by the inequality $h(g,f)\leq \lambda_{g,f}$, with equality
if and only if $g$ has constant curvature and $f$ is a constant function (for the details see \cite{Paternain:2009} and references therein). 
Therefore, what we have to show is that our assumption implies that $\lambda_{g,f}=h(g,f)$. 

Since $(g,\lambda_n f)$ is Zoll, Corollary 2.13 in \cite{Benedetti:2018c} implies that $K_{\lambda_n f}>0$. Thus,
$$\mathcal H(g,\lambda_n f ) =  \frac{A^2}{2\chi(\Sigma)} \, K_{\lambda_n f} <0,\quad \forall n\in \N.$$
By the monotonicity of $\lambda \mapsto \mathcal H(g,\lambda f)$ we deduce that $\mathcal H(g,\lambda f)<0$ for every $\lambda > h(g,f)$, and this implies that $h(g,f)\geq \lambda_{g,f}$,
thus completing the proof.
\end{proof}

%\begin{qst} Does the rigidity statement from Theorem \ref{thm:main} continue to hold for other certain noncompact surfaces, e.g. cylinders?
%\end{qst}

%%%%%%%%%%%

\section{On Conjecture (Z) for $\Sigma =S^2$}
\label{section:4}

The goal of this section is to prove some partial results about Conjecture (Z) on $\Sigma =S^2$. In particular,
we prove that if $(g,f)$ on $S^2$ is such that $(g,\lambda f)$ is Zoll for $\lambda$ sufficiently small then 
$g$ must ``generically'' be Zoll, and that genericity can be dropped when restricting to rotationally invariant magnetic systems.

\begin{prop}
Let $(g,f)$ be a magnetic system on $S^2$ such that $(g,\lambda_n f)$ is Zoll for some sequence
$\lambda_n \downarrow 0$. Then $g$ must ``generically'' 
be Zoll. 
\label{prop:genericallyzoll}
\end{prop}

In order to explain the sentence ``$g$ must generically be Zoll'' we reformulate Proposition \ref{prop:genericallyzoll} as follows: 
if the metric $g$ possesses a non-degenerate closed geodesic (which is well-known to be a generic property) 
then there exists some $\lambda_-=\lambda_-(g,f)>0$ such that for all $\lambda <\lambda_-$
the magnetic system $(g,\lambda f)$ cannot be Zoll.

Thus, let $(g,f)$ be a fixed magnetic system such that $(g,\lambda_n f)$ is Zoll for some sequence $\lambda_n\downarrow 0$. We set
 $$[f] := \frac{1}{\text{Area}(S^2,g)} \int_{S^2} f\, \mu_g$$
 to be the average of $f$. Clearly, we can suppose $[f]\geq 0$. 
 For every $n\in \N$ and every $m\in\N_0$ we define
\begin{align*}
 \mathcal E_n^m & := \Big \{ \text{(prime) closed } (g,\lambda_n f)\text{-geodesics with precisely } m \ \text{self-intersections}\Big \}, \\
  \mathcal E_n & := \bigcup_{m\in \N_0}\ \mathcal E_n^m.
  \end{align*}
In particular, $\mathcal E_n^0$ is the set of embedded closed $(g,\lambda_n f)$-geodesics. 
Here self-intersections are counted with multiplicity, meaning that, if there are more than two branches of the 
$(g,\lambda_n f)$-geodesic $\gamma$ intersecting transversally at a point, 
then we slightly perturb $\gamma$ so that the resulting curve has only simple transversal intersections and define the multiplicity of the 
original intersection point as the minimal number of simple self-intersections arising after perturbation. Notice that, unlike geodesics, 
$(g,\lambda_n f)$-geodesics might have self-tangencies; each self-tangency should count one in the total multiplicity.  

The first step towards the proof of Proposition \ref{prop:genericallyzoll} is to show that it is possible to find an upper bound on
the length of (closed) $(g,\lambda_n f)$-geodesics which depends only on the number of self-intersections. This property holds certainly 
true for general magnetic systems (i.e. independently of the fact that the system be Zoll or not), but we don't need it here in its full generality. 

\begin{lem} For every $m\in\N_0$ there exists $c=c(m)>0$ such that for all 
$n\in \N $ and all $\gamma\in \mathcal E_n^m$ we have 
$$\ell_g(\gamma) \leq c.$$
\label{lem:boundonlength}
\end{lem} 
\vspace{-7mm}
\begin{proof}
By the systolic inequality for magnetic systems close to a Zoll one (see \cite{Benedetti:2018c}) 
we have that, for every $n\in \N$ and for every (closed) $(g,\lambda_n f)$-geodesic $\gamma$,
\begin{align*}
\ell_g(\gamma) + \int_{\mathbb D^2} C_\gamma^* (\lambda_n f \mu_g ) &= \frac{2\pi}{\lambda_n [f] + \sqrt{\lambda_n^2[f]^2 + \frac{2\pi}{\text{Area}(S^2,g)}}} \\
&\leq \sqrt{2\pi \text{Area}(S^2,g)},\end{align*}
where $\ell_g(\gamma)$ denotes the length of $\gamma$ and $C_\gamma:\mathbb D^2\to S^2$ is any admissible capping disk for $\gamma$. Since for any
$\gamma\in \mathcal E_n^0$ the admissible capping disk can be chosen to be embedded, we obtain for all $n\geq n_0$
\begin{align*}
\ell_g(\gamma) &\leq  \sqrt{2\pi \text{Area}(S^2,g)} - \left | \int_{\mathbb D^2} C_\gamma^* (\lambda_n f \mu_g ) \right |\\
			&\leq  \sqrt{2\pi \text{Area}(S^2,g)} + \lambda_n \|f\|_\infty \text{Area}(S^2,g) \\
			&\leq  \sqrt{2\pi \text{Area}(S^2,g)} + \lambda_1 \|f\|_\infty \text{Area}(S^2,g)\\
			&=: c(0).
\end{align*}

For $m\in \N$ the proof goes along the same lines, but employs a finer estimate of the second integral.
For that we need to recall the definition of admissible capping disks from \cite{Benedetti:2018c}. 
For a Zoll magnetic system $(\tilde g,\tilde f)$ the tangential lift $(\gamma,\dot \gamma)$ of a (closed) 
$(\tilde g,\tilde f)$-geodesic is freely homotopic to the fibre of the unit tangent bundle; therefore, we can find a homotopy $\Gamma$ between the fibre and 
$(\gamma,\dot \gamma)$. Admissible capping disks arise now by projecting such homotopies to $S^2$ under the bundle projection.
It can be shown that 
\begin{equation}
\int_{\mathbb D^2} C_\gamma^* (\tilde f \mu_{\tilde g} )
\label{eq:capping}
\end{equation}
does not depend on the choice of the homotopy (hence, on the admissible capping disk).

Let now $\gamma\in \mathcal E_n^m$, and let $(\gamma,\dot \gamma)$ be its tangential lift. Up to an arbitrary small perturbation 
we can assume that all self-intersections of $\gamma$ are simple and that there are no self-tangencies (this perturbation will change the  
value of the integral in \eqref{eq:capping} by a small constant, and hence it is for all estimates not relevant).  
Since the tangential lift of 
$\gamma$ is freely homotopic to the fibre, we can find $t_1>t_0$ such that 
$$q:=\gamma(t_0)=\gamma(t_1), \quad \dot \gamma(t_0)\neq \pm \dot \gamma(t_1),$$
and  $\gamma_0:=\gamma|_{[t_0,t_1]}$ is an embedding (in particular, $\gamma_0$ has turning number $\pm 1$ in 
$S^2\setminus \{p\}$, for some $p\in S^2$). 
We fix $\epsilon >0$ small and pick a homotopy $h$ from $(\gamma_\epsilon := \gamma|_{[t_0-\epsilon,t_1+\epsilon]},\dot \gamma_\epsilon)$ to 
$(\nu_\epsilon,\dot{\nu}_\epsilon)$ that fixes $q$ and (a neighborhood of) the endpoints of $\gamma_\epsilon$, such that $\nu_\epsilon$ 
is (again) an embedding outside $q$ with parallel tangent vectors at $q$. This homotopy will clearly add some bounded quantity to  the  
value of the integral in \eqref{eq:capping}. 
Now we can homotope the tangential lift of the part of $\nu_\epsilon$ that starts from and ends in $q$ with the fibre in $SS^2$. This 
yields a contribution to the integral in \eqref{eq:capping} which is given by the
area of the disk bounded by the considered part of $\nu_\epsilon$ and hence is in absolute value smaller than $\lambda_n \|f\|_\infty$Area$(S^2,g)$ (see Figure \ref{f:homotopy}).
\begin{figure}[h]
\begin{small}
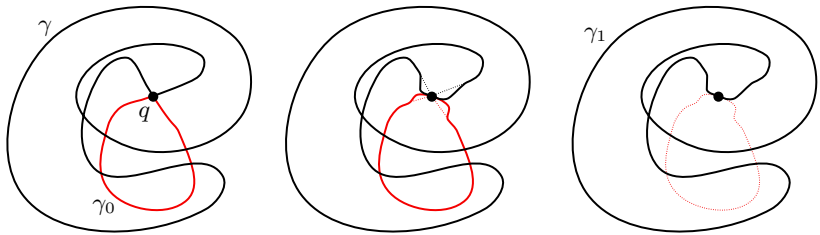 
\caption{Resolving one intersection point for $\gamma$ yields a contribution to \eqref{eq:capping} smaller than $\lambda_n \|f\|_\infty \text{Area}(S^2,g)$.}
%\textbf{(b)} The curve $\gamma_i'$.}
\label{f:homotopy}
\end{small}
\end{figure}

\noindent Therefore, up to the constants arising after perturbation we have that 
$$\left| \int_{\mathbb D^2} C_\gamma^* (\lambda_n f \mu_g )\right | \leq \lambda_n \|f\|_\infty \text{Area}(S^2,g) +\left | \int_{\mathbb D^2} C_{\gamma_1}^* (\lambda_n f \mu_g )\right |,$$
where $\gamma_1$ is the smooth loop given by the concatenation of $\gamma|_{[0,1]\setminus [t_0-\epsilon,t_1+\epsilon]}$ with the remaining 
portion of $\nu_\epsilon$. By construction $\gamma_1$ has one self-intersection less than $\gamma$. Repeating the same 
procedure recursively we obtain a homotopy from $(\gamma,\dot \gamma)$ to some $j$-th iterate of the fibre, $j\leq m$, and then we still have to 
homotope the ``iterated fibre'' to the simple fibre. The first procedure will contribute (up to constants) to the integral with at most $m$-times 
$\lambda_n\|f\|_\infty$Area$(S^2,g)$, whereas the last homotopy yields at most another $j$-times $\lambda_n\|f\|_\infty$Area$(S^2,g)$. Combining all these estimates we obtain that 
\begin{equation*}\left| \int_{\mathbb D^2} C_\gamma^* (\lambda_n f\mu_g )\right | \leq  (2m+1)\lambda_n \|f\|_\infty \text{Area}(S^2,g).\qedhere \end{equation*}
\end{proof}

\vspace{2mm}

Using the bound given by Lemma \ref{lem:boundonlength} we show now that we can generalize to the magnetic setting the
following fact for Riemannian metrics on $S^2$ all of whose geodesics are closed: 
the set of embedded closed geodesics of a Riemannian metric on $S^2$ all of whose geodesics are closed 
is open and close. Notice that this is one main ingredient in the proof of Berger's conjecture for $S^2$; see \cite{Gromoll:1981}.

\begin{lem}
For all $m\in \N_0$ there is $n_0=n_0(m)\in \N$ such that for all $n\geq n_0$ the set $\mathcal E_n^m$ is open and closed in $\mathcal E_n$.
\label{lem:openandclose}
\end{lem}
\begin{proof}
Fix $m\in \N_0$ and denote by $g_{TS^2}$ the Sasaki metric on $TS^2$. For all $T>0$ there exists $\delta=\delta(T)>0$ such that\footnote{Such an inequality might not be true if one does not fix $T$, as there might be e.g. geodesics whose tangential lift is dense.}
$$\text{dist}_{g_{TS^2}} \big ( (q,-v) , \phi^t_g (q,v) \big ) \geq \delta, \quad \forall t\in [0,T], \ \forall (q,v)\in SS^2,$$
where $\phi_g^t$ denotes the geodesic flow. By Lemma \ref{lem:boundonlength}, we can find $c=c(m)>0$ such that, for every $n\in \N$, every 
$\gamma\in \mathcal E_n^m$ has length (hence period) less than $c$. 
Since $(g,\lambda_n f)$-geodesics $C^\infty_{\text{loc}}$-converge to geodesics for $\lambda\to 0$, we find $n_0=n_0(m)\in \N$ such that for all $n\geq n_0$ 
$$\text{dist}_{g_{TS^2}} \big (\phi^t_g (q,v), \phi^t_{g,\lambda_n f}(q,v) \big ) <\frac \delta 2, \quad \forall t\in [0,c], \ \forall (q,v)\in SS^2,$$
where $\phi^t_{g,\lambda_n f}$ denotes the flow on $SS^2$ induced by the magnetic system $(g,\lambda_n f)$ via \eqref{prescribedcurvature}.
Therefore, for all $t\in [0,c]$, all $(q,v)\in SS^2$, and all $n\geq n_0$, we obtain 
\begin{align*}
\text{dist}_{g_{TS^2}} \big ( (q,-v) , \phi^t_{g,\lambda_n f} (q,v) \big )  & \geq \text{dist}_{g_{TS^2}} \big ( (q,-v) , \phi^t_g (q,v) \big ) - \text{dist}_{g_{TS^2}} \big (\phi^t_g (q,v), \phi^t_{g,\lambda_n f}(q,v) \big ) \geq \frac{\delta}{2},
\end{align*}
that is, $(g,\lambda_n f)$-geodesics cannot have self-tangencies before time $c$. Hence, in particular elements in $\mathcal E_n^m$ cannot have 
self-tangencies for all $n\geq n_0$.  This implies that $\mathcal E_n^m$ is 
open and closed for every $n\geq n_0$. Indeed, openness follows immediately from the fact that any sequence of $(g,\lambda_n f)$-geodesics which is 
$C^\infty_{\text{loc}}$-converging to $\gamma\in \mathcal E_n^m$ is 
actually $C^\infty$-converging to $\gamma$ (the periods vary smoothly for a Zoll magnetic system)
combined with the fact that having $m$ transversal self-intersecions is a $C^1$-open condition. 
On the other hand, the limit $\gamma$ of a sequence $\{\gamma_\ell\}\subset \mathcal E_n^m$ cannot have more self-intersections that any of the $\gamma_\ell$'s, 
and since $(g,\lambda_n f)$ is Zoll we also have that $\gamma$ cannot 
have less self-intersections than any of the $\gamma_\ell$'s (since $\gamma$ cannot have self-tangencies, the number of self-intersections can only jump if there is a 
jump in the periods). 
\end{proof}

\begin{cor}
The following holds: 
\begin{enumerate}[a)]
\item Let $m\in\N_0$ be fixed. If for some $n\geq n_0$, where $n_0$ is given by Lemma \ref{lem:openandclose}, 
we have that $\mathcal E_n^m \neq \emptyset$, then all $(g,\lambda_n f)$-geodesics have precisely $m$ (transversal) self-intersections.  
\item If $\mathcal E_{n_k}^{m_k}\neq \emptyset$ for some sequence $n_k\to +\infty$ and some bounded sequence $m_k$ then $g$ is Zoll.
\end{enumerate}
\label{cor:Zoll}
\end{cor}
\begin{proof}
The first assertion is clear. Suppose now that there is a non-closed geodesic 
$t\mapsto \phi^t_g(q,v)$. Then from the $C^\infty_{\text{loc}}$-convergence of $(g,\lambda_n f)$-geodesics to geodesics for $n \to +\infty$ 
we deduce 
$$\ell_g (t\mapsto \phi^t_{g,\lambda_n f}(q,v) ) \to +\infty,$$
thus contradicting the assumption. Therefore, every geodesic in $(S^2,g)$ is closed, and hence $g$ is Zoll. 
\end{proof}

\begin{proof}[Proof of Proposition \ref{prop:genericallyzoll}]
Suppose that $(g,f)$ is such that $(g,\lambda_n f)$ is Zoll for some sequence $\lambda_n\downarrow 0$, 
and that $g$ has a non-degenerate closed geodesic $\gamma$ (in particular $g$ is not Zoll). 
Then, following \cite{Ginzburg:1994} we find for all $n$ large enough a closed $(g,\lambda_n f)$-geodesic in a neighborhood of $\gamma$, 
which will therefore have the same number of self-intersections as $\gamma$. By Corollary \ref{cor:Zoll}, this implies that $g$ is 
Zoll, a contradiction.
\end{proof}

We believe that Proposition \ref{prop:genericallyzoll} holds also non generically. Indeed, Lemmas
\ref{lem:boundonlength} and \ref{lem:openandclose} do not require any non-degeneracy assumption, and in order to apply
Corollary \ref{cor:Zoll} we only need the existence of a closed geodesic which persists under magnetic perturbations. 
However, to our best knowledge nothing like Ginzburg's result \cite{Ginzburg:1994} is known for degenerate closed geodesics. 
Therefore, in order to extend Proposition \ref{prop:genericallyzoll} one has either to use other methods or to first better understand 
how degenerate closed geodesics behave under magnetic perturbations, which is in any case a problem of independent interest. 

We finish this section showing that Proposition \ref{prop:genericallyzoll} can be improved in the rotationally symmetric setting by dropping the genericity assumption.
In the statement of the next lemma we denote with $(\theta,\varphi)\in (0,\pi)\times \R$ spherical coordinates on $S^2$.  

\begin{lem}
Let $(g,f)$ be a rotationally symmetric magnetic system on $S^2$ such that $(g,\lambda_n f)$ is Zoll for some sequence $\lambda_n \to 0$. 
Then $g$ is a Zoll metric of revolution.
\label{lem:alwayszoll}
\end{lem}
\begin{proof}
It is well-known (cf. \cite[Section 5]{Benedetti:diss}) that 
rotationally invariant magnetic systems admit a first integral $I$ whose critical points correspond to \textit{latitudes}, that is, 
periodic orbits with constant $\theta$-coordinate. Since $I$ has at least two critical points, this yields the existence 
of at least two embedded periodic orbits; see \cite[Proposition 5.9]{Benedetti:diss}. The claim follows now from Corollary \ref{cor:Zoll}.
\end{proof}

%%%%%%%%%%%%%%%

\section{Persistence of stable waists.}
\label{section:5}
We recall that a \textit{waist} for a Riemannian metric $g$ on $S^2$ is a closed geodesic which is a local minimizer of the length, or, equivalently,
of the free period action functional 
$\A$ given by \eqref{freeperiod1}. 
We call a waist $\gamma=(\Gamma,T)$ \textit{stable} if there exists a bounded neighborhood $\U \subset H^1(\T,S^2)\times (0,+\infty)$ 
of the critical circle $S^1 \cdot \gamma$ such that 
\begin{equation}
\inf_\U \A = \A(\gamma),\quad \text{and}\ \ \inf_{\partial \U} \A > \A(\gamma).
\label{stableneighborhood}
\end{equation}

\begin{thm}
Let $g$ be a metric on $S^2$ possessing a stable waist. Then for every magnetic function $f$ there exists $\Lambda =\Lambda(g,f)>0$ such that for every $\lambda\in (0,\Lambda)$ 
the magnetic pair $(g,\lambda f)$ has a closed magnetic geodesic $\gamma_{\lambda}$ which is a local minimizer of the free-period action functional $\A^\lambda$ given by \eqref{freeperiod2}. 
\label{thm:waistthm}
\end{thm} 

Combining Theorem \ref{thm:waistthm} with either Corollary \ref{cor:Zoll} or the argument in the first paragraph of the proof of Proposition \ref{prop:fconstant} we can confirm the validity of Conjecture (Z) in case of metrics on $S^2$ admitting a stable waist.

\begin{cor}\label{cor:stable_waist}
Let $(g,f)$ be a magnetic system on $S^2$, where $g$ is a metric possessing a stable waist. Then there exists $\lambda>0$ such that $(g,\lambda f)$ is not Zoll. \qed
\end{cor}

Notice that a non-stable waist might disappear after an arbitrarily small perturbation of the metric. Take for instance a smooth sphere of revolution in 
$\R^3$ with profile function $\varphi:[-1,1] \to (0,1]$ such that $\varphi\equiv 1$ on $(-\epsilon,\epsilon)$ for some $\epsilon>0$. 
The closed geodesics given by parallels $\{z=t\}$, $t\in (-\epsilon,\epsilon)$, are indeed waists, 
but we can find arbitrarily small perturbations $\chi$ of $\varphi$ such that $\varphi + \chi$ has a unique maximum at $0$ and no other critical points.
The corresponding surface of revolution has no waists, and the closed geodesic given by $\{z=0\}$ is of mountain pass type.
Therefore, an analogue of Theorem \ref{thm:waistthm} for such waists is hopeless. 

In case the stable waist in Theorem \ref{thm:waistthm} is strict (isolated in the language of Section \ref{section:2}) the proof is identical to the proof of 
Proposition \ref{prop:persistencetorus} and will be omitted. Therefore, hereafter we can assume that all stable waists are non-isolated. 
We choose such a non-isolated stable waist $\gamma$ and fix a neighborhood $\U$ 
of $S^1\cdot \gamma$ satisfying \eqref{stableneighborhood}.
The main difficulty that we have to face here is given by the fact that, unlike in Proposition \ref{prop:persistencetorus}, we cannot expect the compact set 
$$K_\U:=\{\nu \in \U\, |\, \A(\nu)=\A(\gamma)\}$$ 
to be mapped into a proper compact subset of $S^2$ by the evaluation map, for local minimizers which are not global need not have disjoint image. 
Observe that the compactness of $K_\U$ follows from \eqref{stableneighborhood} combined with the fact that $\A$ satisfies the Palais-Smale condition.

In fact, the key part of the proof of Theorem \ref{thm:waistthm} is to show that we can find a possibly smaller neighborhood $\V\subset \U$ 
of $S^1\cdot \gamma$ such that \eqref{stableneighborhood} still holds and such that the closure of $\text{ev}(\V)$ is a proper subset of $S^2$. Once this is done, the proof becomes
identical to the one in the isolated case. 
\begin{figure}[h]
\begin{small}
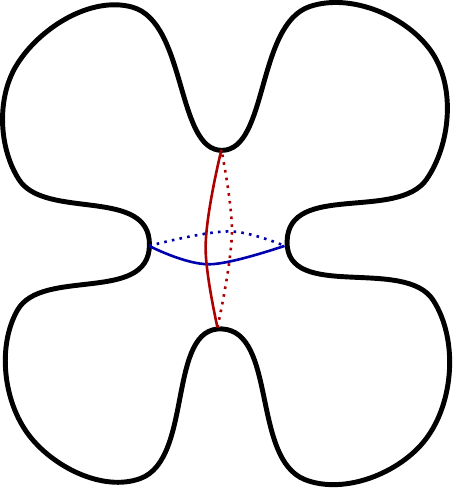 
\caption{The images of two waists need not be disjoint.}
%\textbf{(b)} The curve $\gamma_i'$.}
\label{f:disjointness}
\end{small}
\end{figure}

As a first step we show that $\gamma$ is embedded, and that all $\nu\in K_\U$ which are sufficiently close to $\gamma$ must have pairwise disjoint image. For the next lemma we actually do not even need that $\gamma$ is stable.

\begin{lem}\label{lem:local_intersection}
Let $\gamma$ be a non-strict local minimizer of $\A$, and let $\U$ be a neighborhood of $S^1\cdot \gamma$ as 
in \eqref{stableneighborhood}. Then there exists a neighborhood $\W_\gamma\subset \U$ of $S^1\cdot \gamma$ 
such that all $\gamma_1\neq \gamma_2\in K_\U \cap \W_\gamma$ have disjoint image unless they belong to the same critical circle.
In particular, every non-strict local minimizer must be embedded.
\label{lem:nearbydisjoint}
\end{lem}

\begin{proof}
We first show that if $\{\gamma_n\}\subset K_\U$ is a sequence such that $\gamma_n\to \gamma$
in $H^1$ (and hence in $C^\infty$) then eventually $\gamma_n$ and $\gamma$ must have disjoint images unless they belong to the same critical circle. Indeed, suppose by contradiction that (up to taking a subsequence) $\gamma_n$ and $\gamma$ intersect transversally for every $n\in \N$. Since the image of $\gamma_n$ is contained in an annular region around the image of $\gamma$, we 
have that $\gamma_n$ and $\gamma$ must intersect at least twice (even though the intersection point might be unique, as one easily sees by taking two figure 
eight curves in the plane intersecting at zero). Thus we can find $t_0^n<t_1^n$ and $s_0^n<s_1^n$ such that 
$$\gamma(t_0^n) = \gamma_n(s_0^n),\quad \gamma(t_1^n) = \gamma_n (s_1^n)$$
and define the curves 
$$\eta_{n,1} := \gamma |_{[t_0^n,t_1^n]} \# \gamma_n |_{[s_1^n,s_0^n]}, \quad \eta_{n,2} := \gamma_n |_{[s_0^n,s_1^n]}\#\gamma |_{[t_1^n,t_0^n]}.$$
Now the same line of arguments as in the end of the proof of Proposition \ref{prop:persistencetorus} yields a contradiction.
%We readily see that 
%$$\A(\eta_{n,1}) + \A(\eta_{n,2}) = \A(\gamma)+ \A(\gamma_n) = 2 \A(\gamma),$$
%so that without loss of generality $\A(\eta_{n,1}) \leq \A(\gamma)$. Since $\gamma$ and $\gamma_n$ are $C^\infty$-close, $\eta_{n,1}$ and 
%$\gamma$ are $H^1$-close, hence in particular $\eta_{n,1}\in \U$ for $n\in \N$ large enough. Therefore $\eta_{n,1}\in K_\U$, and hence $\eta_{n,1}$ is smooth. 
%However, this is possible if and only if $\gamma$ and $\gamma_n$ have the same image. 
%From this we readily see that $\gamma$ must be embedded. Indeed, if $\gamma$ had a self-intersection point, then $\gamma_n$ and $\gamma$ could not 
%have disjoint image for $n$ large. Moreover, since being embedded is a $C^1$-open condition, all $\gamma_n$'s which are sufficiently close to $\gamma$ 
%must be embedded. 

Now suppose that we can find sequences $\{\gamma_n\}$ and $\{\gamma_n'\}$ in $K_\U$ that converge
to $\gamma$ and such that $\gamma_n$ and $\gamma_n'$ intersect for every $n\in\N$. By the Jordan curve theorem $\gamma_n$ and $\gamma_n'$ must intersect 
at least twice. Arguing as above, we see that $\gamma_n$ and $\gamma_n'$ are eventually the same geometric curve. 
\end{proof}

In particular, the evaluation map maps $\W_\gamma$ to a proper subset of $\Sigma$. However, Lemma \ref{lem:nearbydisjoint} is not quite enough to complete the proof of Theorem \ref{thm:waistthm} since it is a priori not clear that the infimum of $\A$ over  
$\partial \W_\gamma$ is strictly larger than $\A(\gamma)$. Recall indeed that all elements in $K_\U$ are non-isolated; for the same reason, we see that it does not 
make much sense to replace $K_\U$ with the connected component of $\gamma$ in $K_\U$. We will overcome this difficulty proving   
an analogue of Lemma \ref{lem:nearbydisjoint} for a suitable subset of $K_\U$ which can be seen as a generalized connected component of $\gamma$ in $K_\U$, and which 
we now define.

Lemma \ref{lem:nearbydisjoint} implies that for all $\nu\in K_\U$ we find a neighborhood $\W_\nu$ such that every $\gamma_1,\gamma_2\in \W_\nu \cap K_\U$ 
either have disjoint image or they are the same geometric curve. Since $\{\W_\nu \, |\,\nu \in K_\U\}$ is an open covering of $K_\U$ and $K_\U$ is compact,
we can find $\nu_1,...,\nu_\ell\in K_\U$ such that 
$$K_\U\subset \bigcup_{i=1}^\ell \W_{\nu_i}.$$
By Lebesgue's number lemma we see now that there exists $\rho>0$ such that, for all $\nu\in K_\U$, every two $\gamma_1,\gamma_2 \in B_\rho(S^1\cdot \nu)\cap K_\U$ 
either have disjoint image or are the same geometric curve.

\begin{dfn}
Fix $\delta<\rho$. Two elements $\gamma_0,\gamma_1\in K_\U$ are said $\delta$-\textit{connected} if there exist $M\in \N$ and a family $\{\eta_i\}_{i=0,...,M}\subset K_\U$
such that $\gamma_0=\eta_0$, $\gamma_1=\eta_M$, and 
$\eta_{i+1}\in B_\delta (S^1\cdot \eta_i)\setminus S^1\cdot \eta_i, \ \forall i=0,...,M-1.$
\label{dfn:deltaconnected}
\end{dfn}

%\begin{lem} There exists some $c>0$ such that for every sufficiently small $\delta>0$ the following holds: If $\eta_1,\eta_2 \in K_\U$ do not intersect and satisfy $\mathrm{dist} (S^1\cdot \eta_1,S^1\cdot \eta_{2}) < \delta$, then any $\eta \in K_\U$ which 
%is completely contained in the thin component of $S^2-(|\eta_1|\cup |\eta_2|)$ satisfies $\mathrm{dist} (S^1\cdot \eta_i,S^1\cdot \eta) < c\delta$, $i=1,2$.
%\label{lem:local_delta}
%\end{lem}

\begin{lem}
%The following hold:
%\begin{enumerate}[a)]
%\item Let $\eta\in \U$ be a local minimizer of $\A$ with the same action of $\gamma$ which is $\delta$-connected to $\gamma$ for all $\delta>0$. 
%Then $\gamma$ and $\eta$ are either the same geometric curve or have disjoint image. 
%\item 
For $\delta<\rho$ small enough, every $\nu\in K_\delta(\gamma)$,
$$K_\delta(\gamma) := \{\nu \in K_\U \ |\ \nu,\gamma \ \text{are} \ \delta\text{-connected}\},$$ 
does not intersect $\gamma$, unless it belongs to $S^1 \cdot \gamma$. 
%\end{enumerate}
\label{lem:deltaconnected}
\end{lem}
\begin{proof}
We first observe that the $C^\infty$-dependence of geodesics on the initial conditions, combined with the fact that $K_\U$ is compact, 
yields that there exists $C>0$ such that the following holds for all $\gamma_1,\gamma_2\in K_\U$:
\begin{align}
%\text{dist}_{H^1}(\gamma_1,\gamma_2)<\epsilon \ &\Rightarrow \ \text{dist}_{g_{TS^2}} ((\gamma_1(0),\dot \gamma_1(0)), ((\gamma_2(0),\dot \gamma_2(0))<c\epsilon,\\
\text{dist}_{g_{TS^2}} ((\gamma_1(0),\dot \gamma_1(0)), ((\gamma_2(0),\dot \gamma_2(0))<\epsilon \ &\Rightarrow \ \dist_{H^1} (\gamma_1,\gamma_2)<C \epsilon.
\label{eq:sasakih1}
\end{align}

Choose now $\delta<\rho$ and let $\nu \in K_\delta(\gamma)$. Our aim is to show that $\nu$ and $\gamma$ have disjoint image if they are not the same geometric curve. 
Thus, suppose that $\nu(\cdot)\cap \gamma(\cdot)\neq \emptyset$, and let $M\in \N$ and $\{\eta_i\}_{i=1}^M$ be a sequence as in Definition \ref{dfn:deltaconnected}. 
By construction, $\eta_i$ and $\eta_{i+1}$ have disjoint image for every $i=0,...,M-1$. We now consider a subfamily $\{\eta_j,...,\eta_{j+\ell}\}$ of $\{\eta_i\}$ which is minimal 
among the subfamilies (of the same form) that satisfy the following properties:
\begin{enumerate}[i)]
\item $\eta_j$ and $\eta_{j+\ell}$ intersect.
\item $\eta_{j+r}$ and $\eta_{j+s}$ have disjoint image for all $r< s$ with $r\neq 1$ and $s\neq \ell$.
\end{enumerate} 

Let $\eta_j(0)=\eta_{j+\ell}(t)$ be an intersection point. Since $\eta_{j+\ell}$ and $\eta_{j+1}$ have disjoint image,
we deduce that 
$$\text{dist}_{g_{TS^2}} ((\eta_j(0),\dot \eta_j(0)), ((\eta_{j+\ell}(t),\dot \eta_{j+\ell}(t))<\epsilon$$
for some $\epsilon>0$ which, by compactness of $K_\U$, only depends on $\delta$ and goes to zero as $\delta \downarrow 0$. Then, by \eqref{eq:sasakih1}
$$\dist_{H^1} (\eta_j, \eta_{j+\ell}) < C\epsilon,$$
and hence $\eta_j$ and $\eta_{j+\ell}$ are the same geometric curve, provided we chose $\delta<\rho$ so small that $C\epsilon<\rho$. Therefore, 
the subfamily $\{\eta_j,...,\eta_{j+\ell}\}$ can be removed from $\{\eta_i\}$. Applying this procedure recursively we obtain that $\gamma$ and $\nu$ are the same geometric curve.
\end{proof}

\begin{lem}
For $\delta>0$ as in Lemma \ref{lem:deltaconnected} 
we can find a bounded neighborhood $\V\subset \U$ of $S^1\cdot \gamma$ such that 
$$\inf_{\partial \V}\A > \A(\gamma), \quad \text{and}\quad  K_\V := K_\U \cap \V = K_\delta(\gamma).$$
\label{lem:final}
\end{lem}
\vspace{-6mm}
\begin{proof}
Observe preliminarly that, if $\nu\in K_\U$ is not $\delta$-connected to $\gamma$, then every $\mu \in B_{\alpha_\nu}(\nu)\cap K_\U$ is not $\delta$-connected
to $\gamma$, where 
$$\delta \leq \alpha_\nu := \inf_{M\in \N} \inf_{\eta_0,...,\eta_{M}} \text{dist} (\eta_M,\nu),$$
and $\{\eta_i\}_{i=0,...,M}$ is any sequence with $\gamma=\eta_0$ and $\eta_{i+1}\in B_\delta( S^1\cdot \eta_i) \setminus S^1\cdot \eta_i$ for all $i.$
For $\nu\in K_\U$ we consider $B_{\alpha_\nu}(\nu)$ if $\nu$ is not $\delta$-connected with $\gamma$, and $B_\delta(\nu)$ otherwise.
By compactness of $K_\U$ we find $\mu_1,...,\mu_\ell\in K_\U$ which are not $\delta$-connected with $\gamma$ and $\mu_{\ell+1},..., \mu_{\ell+r}\in K_\U$
which are $\delta$-connected with $\gamma$ such that 
$$K_\U \subset \bigcup_{i=1}^\ell B_{\alpha_{\mu_i}}(\mu_i) \ \cup \ \bigcup_{i=\ell +1}^r B_\delta(\mu_{\ell +i}).$$
For the sake of simplicity we assume that all balls are entirely contained in $\U$ (otherwise we can work with the intersection of the balls with $\U$). 
Notice that any $\mu$ which lies in 
$$K_\U' := K_\U \cap \left (  \bigcup_{i=1}^\ell \overline{B_{\alpha_{\mu_i}}(\mu_i)} \setminus  \bigcup_{i=1}^\ell B_{\alpha_{\mu_i}}(\mu_i)\right )$$
is $\delta$-connected to $\gamma$, that is, elements in $K_\U$ which are on the boundary of some $B_{\alpha_{\mu_i}}(\mu_i)$ but not in the 
interior of some other $B_{\alpha_{\mu_j}}(\mu_j)$ are necessarily $\delta$-connected with $\gamma$. 

If $ K_\U' =\emptyset$ then we take as $\V$ the connected component of $\U\setminus \cup_{i=1}^\ell \overline{B_{\alpha_{\mu_i}}(\mu_i)}$ 
that contains $\gamma$. Otherwise, for every $\mu\in K_\U'$ we consider the open ball $B_{\delta/2}(\mu)$. Again by compactness we 
can find finitely many $\lambda_1,...,\lambda_s\in K_\U'$ such that 
$$K_\U' \subseteq  \bigcup_{i=1}^s B_{\delta/2}(\lambda_i).$$
We now observe that by construction 
$$ \partial \Big  (\bigcup_{j=1}^s  B_{\delta/2}(\lambda_j)\Big ) \cap  \Big (\bigcup_{i=1}^\ell \overline{B_{\alpha_{\mu_i}}(\mu_i)}\Big )$$
cannot contain elements of $K_\U$. Therefore, for
$$\mathcal A :=  \left (  \bigcup_{i=1}^\ell \overline{B_{\alpha_{\mu_i}}(\mu_i)} \right ) \setminus \bigcup_{i=1}^s B_{\delta/2}(\lambda_i) $$
we have  $ \inf_{\partial\mathcal A} \A > \A(\gamma)$. 
The assertion follows now taking the component $\V$ of $\U\setminus \mathcal A$ that contains $\gamma$. 
\end{proof}

\begin{proof}[Proof of Theorem \ref{thm:waistthm}]
Let $\V$ be a neighborhood of the non-isolated stable waist $\gamma$ as in Lemma \ref{lem:final}. Then, since the image of any $\nu \in K_\V=K_\delta(\gamma)\setminus S^1\cdot \gamma$ 
is entirely contained in one of the two disks in which $S^2$ is divided by $\gamma$, and 
since all geodesics in $K_\V$ have the same length, Lemma \ref{lem:deltaconnected} implies that
$K_\V$ is mapped by the evaluation map into a subset of $S^2$ with proper closure, and hence up to shrinking $\V$ further we also have that the closure of the image 
of $\V$ under the evaluation map is a proper subset of $S^2$.
\end{proof}

Repeating the proof above word by word we obtain the following result on the persistence 
of contractible stable waists on arbitrary closed surfaces.

\begin{thm}
\label{thm:finish}
Let $g$ be a metric on a closed surface $\Sigma$ possessing a contractible stable waist. Then for every magnetic function $f$ there exists $\Lambda =\Lambda(g,f)>0$ 
such that for every $\lambda\in (0,\Lambda)$ 
the magnetic pair $(g,\lambda f)$ has a contractible closed $(g,\lambda f)$-geodesic $\gamma_{\lambda}$ which is a local minimizer 
of the functional $\A^\lambda$ in \eqref{freeperiod2}. \qed
\end{thm} 

It would be very interesting to see whether analogues of Theorem \ref{thm:finish} and Proposition \ref{prop:persistencetorus}
hold for higher dimensional manifolds, in particular for those whose fundamental group is ameanable, as in this case 
the Ma\~n\'e critical value $c(g,f)$ is infinite whenever the magnetic function does not integrate to zero \cite{Merry:2010}.

\bibliography{_biblio}

\begin{thebibliography}{10}

\bibitem{Abbondandolo:2013is}
A.~Abbondandolo.
\newblock Lectures on the free period {L}agrangian action functional.
\newblock {\em J. Fixed Point Theory Appl.}, 13(2):397--430, 2013.

\bibitem{Abbondandolo:2018}
A.~Abbondandolo, B.~Bramham, U.~Hryniewicz, and P.~Salomao.
\newblock Sharp systolic inequalities for {R}iemannian and {F}insler spheres of
  revolution.
\newblock arXiv:1808.06995, 2018.

\bibitem{Abbondandolo:2014rb}
A.~Abbondandolo, L.~Macarini, M.~Mazzucchelli, and G.~P. Paternain.
\newblock Infinitely many periodic orbits of exact magnetic flows on surfaces
  for almost every subcritical energy level.
\newblock {\em J. Eur. Math. Soc. (JEMS)}, 19(2):551--579, 2017.

\bibitem{Agapov:2017}
S.V. Agapov, M.L. Bialy, and A.E. Mironov.
\newblock Integrable magnetic geodesic flows on 2-torus: {N}ew examples via
  quasi-linear system of {P}{D}{E}s.
\newblock {\em Comm. Math. Phys.}, 351:993--1007, 2017.

\bibitem{Arnold:1961}
V.~I. Arnold.
\newblock Some remarks on flows of line elements and frames.
\newblock {\em Dokl. Akad. Nauk SSSR}, 138:255--257, 1961.

\bibitem{Asselle:2015ij}
L.~Asselle and G.~Benedetti.
\newblock Infinitely many periodic orbits in non-exact oscillating magnetic
  fields on surfaces with genus at least two for almost every low energy level.
\newblock {\em Calc. Var. Partial Differ. Equ.}, 54(2):1525--1545, 2015.

\bibitem{Asselle:2014hc}
L.~Asselle and G.~Benedetti.
\newblock The {L}usternik-{F}et theorem for autonomous {T}onelli {H}amiltonian
  systems on twisted cotangent bundles.
\newblock {\em J. Topol. Anal.}, 8(3):545--570, 2016.

\bibitem{Asselle:2015hc}
L.~Asselle and G.~Benedetti.
\newblock On the periodic motions of a charged particle in an oscillating
  magnetic flows on the two-torus.
\newblock {\em Math. Z.}, 286(3-4):843--859, 2017.
\newblock online first.

\bibitem{Asselle:2019a}
L.~Asselle and G.~Benedetti.
\newblock Integrable zoll magnetic systems on the two-torus.
\newblock https://arxiv.org/abs/1909.13821, 2019.

\bibitem{Bangert:1986}
V.~Bangert.
\newblock On the lengths of closed geodesics on almost round spheres.
\newblock {\em Math. Z.}, 191(4):549--558, 1986.

\bibitem{Bangert:1994}
V.~Bangert.
\newblock Busemann functions and monotone twist maps.
\newblock {\em Calc. Var. Partial Differ. Equ.}, 2:49--63, 1994.

\bibitem{Benedetti:diss}
G.~Benedetti.
\newblock {\em The contact property for magnetic flows on surfaces}.
\newblock PhD thesis, University of Cambridge, 2014.

\bibitem{Benedetti:2016}
G.~Benedetti.
\newblock Magnetic {K}atok examples on the two-sphere.
\newblock {\em Bull. London Math. Soc.}, 48(5):855--865, 2016.

\bibitem{Benedetti:2018c}
G.~Benedetti and J.~Kang.
\newblock On a systolic inequality for closed magnetic geodesics on surfaces.
\newblock arXiv:1902.01262, 2018.

\bibitem{Besse}
A.~L. Besse.
\newblock {\em Manifolds all of whose geodesics are closed}.
\newblock Ergebnisse der Mathematik und ihrer Grenzgebiete, 1978.

\bibitem{Bialy:2011}
M.L. Bialy and A.E. Mironov.
\newblock Rich quasi-linear system for integrable geodesic flow on 2-torus.
\newblock {\em Discrete Contin. Dyn. Syst.}, 29:81--90, 2011.

\bibitem{Bialy:2015}
M.L. Bialy and A.E. Mironov.
\newblock Integrable geodesic flows on 2-torus: formal solutions and
  variational principle.
\newblock {\em J. Geom. Phys.}, 81(1):39--47, 2015.

\bibitem{Birkhoff:1966xb}
G.~D. Birkhoff.
\newblock {\em Dynamical systems}.
\newblock With an addendum by Jurgen Moser. American Mathematical Society
  Colloquium Publications, Vol. IX. American Mathematical Society, Providence,
  R.I., 1966.

\bibitem{Contreras:2006yo}
G.~Contreras.
\newblock {The Palais-Smale condition on contact type energy levels for convex
  Lagrangian systems}.
\newblock {\em Calc. Var. Partial Differ. Equ.}, 27(3):321--395, 2006.

\bibitem{Contreras:2004lv}
G.~Contreras, L.~Macarini, and G.~P. Paternain.
\newblock Periodic orbits for exact magnetic flows on surfaces.
\newblock {\em Int. Math. Res. Not.}, (8):361--387, 2004.

\bibitem{Epstein:1972}
D.B.A. Epstein.
\newblock Periodic flows on three-manifolds.
\newblock {\em Ann. of Math. (2)}, 95:66--82, 1972.

\bibitem{Ginzburg:1987lq}
V.~L. Ginzburg.
\newblock New generalizations of {P}oincar{\'e}'s geometric theorem.
\newblock {\em Funktsional. Anal. i Prilozhen.}, 21(2):16--22, 96, 1987.

\bibitem{Ginzburg:1994}
V.~L. Ginzburg.
\newblock {\em On closed trajectories of a charge in a magnetic field. An
  application of symplectic geometry.}, volume~8, pages 131--148.
\newblock Cambridge Univ. Press, 1996.

\bibitem{Gromoll:1981}
D.~Gromoll and K.~Grove.
\newblock On metric on $s^2$ all of whose geodesics are closed.
\newblock {\em Invent. Math.}, 65:175--177, 1981.

\bibitem{Guillemin:1976}
V.~Guillemin.
\newblock The {R}adon transform on {Z}oll surfaces.
\newblock {\em Advances in Math.}, 22(1):85--119, 1976.

\bibitem{Hedlund:1932am}
G.~A. Hedlund.
\newblock Geodesics on a two-dimensional riemannian manifold with periodic
  coefficients.
\newblock {\em Ann. of Math. (2)}, 33(4):719--739, 1932.

\bibitem{Innami:1986}
N.~Innami.
\newblock Families of closed geodesics which distinguish flat tori.
\newblock {\em Math. J. Okayama Univ.}, 28:207--217, 1986.

\bibitem{Jankins}
M.~Jankins and W.~D. Neumann.
\newblock {\em Lectures on {S}eifert manifolds}.
\newblock Brandeis Lecture Notes 2, Brandeis University, 1983.

\bibitem{Katok:1973mw}
A.~B. Katok.
\newblock Ergodic perturbations of degenerate integrable {H}amiltonian systems.
\newblock {\em Izv. Akad. Nauk SSSR Ser. Mat.}, 37:539--576, 1973.

\bibitem{Kerman:1999}
E.~Kerman.
\newblock Periodic orbits of {H}amiltonian flows near symplectic critical
  submanifolds.
\newblock {\em Int. Math. Res. Not.}, (17), 1999.

\bibitem{Kozlov:1989}
V.~V. Kozlov and D.V. Treschev.
\newblock On the integrability of {H}amiltonian systems with toral position
  space.
\newblock {\em Math. USSR Sbornik}, 63(1):121--139, 1989.

\bibitem{Merry:2010}
W.~Merry.
\newblock Closed orbits of a charge in a weakly exact magnetic field.
\newblock {\em Pacific J. Math.}, 47(1):189--212, 2010.

\bibitem{Paternain:2009}
G.~P. Paternain.
\newblock Helicity and the ma{{\~n\'e}} critical value.
\newblock {\em Algebr. Geom. Topol.}, 9:1413--1422, 2009.

\bibitem{Raymond:2015}
N.~Raymond and S.~Vu~Ngoc.
\newblock Geometry and {S}pectrum in 2{D} mangetic wells.
\newblock {\em Ann. Inst. Fourier (Grenoble)}, 65(1):137--169, 2015.

\bibitem{Taimanov:1992sm}
I.~A. Taimanov.
\newblock Closed non-self-intersecting extremals of multivalued functionals.
\newblock {\em Sibirsk. Mat. Zh.}, 33(4):155--162, 223, 1992.

\end{thebibliography}
\bibliographystyle{plain}
\end{document}